\documentclass[reqno]{amsart}
\usepackage{amssymb,amscd,verbatim, amsthm,graphics, color,latexsym,amsmath,multicol}
\usepackage{fancybox}
\usepackage{color}
\usepackage{longtable}

\usepackage[all]{xy}

\newcommand \C[1]{{\mathcal #1}}

\newcommand \wti[1]{{\widetilde {#1}}}

\newcommand \bC{{\mathbb C}}

\newcommand \bH{{\mathbb H}}
\newcommand \bR{{\mathbb R}}
\newcommand \bZ{{\mathbb Z}}

\newcommand \bQ{{\mathbb Q}}

\newcommand\ep{{\epsilon}}

\newcommand\om{{\omega}}
\newcommand\al{{\alpha}}

\newtheorem{theorem}{Theorem}[subsection]

\newtheorem{corollary}{Corollary}[subsection]
\newtheorem{lemma}{Lemma}[subsection]
\newtheorem{proposition}{Proposition}[subsection]
\newtheorem{definition}{Definition}[subsection]
\newtheorem{remark}{Remark}[subsection]

\newtheorem{ftheorem}{Theorem}{\bf}{\it}
\newtheorem{fcorollary}[ftheorem]{Corollary}{\bf}{\it}

\newcommand\Hom{\operatorname{Hom}}
\newcommand\End{\operatorname{End}}
\newcommand\Ind{\operatorname{Ind}}

\newcommand\tr{\operatorname{tr}}

\newcommand\triv{\mathsf{triv}}
\newcommand\sgn{\mathsf{sgn}}
\newcommand\refl{\mathsf{refl}}

\newcommand\sem{\mathsf{ss}}
\newcommand\Irr{\mathsf{Irr}}
\newcommand\good{\mathsf{good}}
\newcommand\cc{\mathsf{cc}}
\newcommand\el{\mathsf{ell}}
\newcommand\res{\mathsf{res}}
\newcommand\rcc{\mathsf{rcc}}
\newcommand\Tr{\mathsf{Tr}}

\def\<{\langle} 
\def\>{\rangle}
\def\i{^{-1}}
\def\s{\sigma}
\def\a{\alpha}
\def\co{\mathcal O}
\def\d{\delta}
\newcommand{\supp}{\text {\rm supp}}

\numberwithin{equation}{subsection}

\begin{document}

\title[The cocenter of graded affine Hecke algebra]{The cocenter of graded affine Hecke algebra and the density
theorem}

\author{Dan Ciubotaru}
        \address[D. Ciubotaru]{Dept. of Mathematics\\ University of
          Utah\\ Salt Lake City, UT 84112}
        \email{ciubo@math.utah.edu}

\author{Xuhua He}
\address[X. He]{Dept. of Mathematics\\Hong Kong University of Science
  and Technology\\Clear Water Bay, Kowloon, Hong Kong}
\email{maxhhe@ust.hk}



%\thanks{}

\begin{abstract}We determine a basis of the (twisted)
  cocenter of graded affine Hecke algebras with arbitrary
  parameters. In this setting, we prove that the kernel of the (twisted) trace map is
  the commutator subspace (Density theorem) and that the image is the
  space of good forms (trace Paley-Wiener theorem). 
\end{abstract}

\maketitle

\section{Introduction}\label{sec:0}
The affine Hecke algebras arise naturally in the theory
of smooth representations of reductive $p$-adic groups. Motivated by
the relation with abstract harmonic analysis for $p$-adic groups (such
as the trace Paley-Wiener theorem and the Density theorem \cite{BDK,Ka,Fl}), as
well as the study of affine Deligne-Lusztig varieties (such as the ``dimension=degree'' theorem \cite[Theorem 6.1]{H99}), it is important to describe the cocenter of affine Hecke algebras, i.e.,
the quotient of the Hecke algebra by the vector subspace spanned by
all commutators. In this paper, we solve the related problem for
the graded affine Hecke algebras introduced by Lusztig \cite{L1}. 

\

To describe the results, let $\bH$ be the
graded Hecke algebra attached to a simple root system $\Phi$ and complex
parameter function $k$, Definition \ref{d:graded}. As a
$\bC$-vector space, $\bH$ is isomorphic to $\bC[W]\otimes S(V)$, where
$W$ is the Weyl group of $\Phi,$ and $S(V)$ is the symmetric algebra
of $V$, the underlying (complex) space of the root system.

Let $\delta$ be
an automorphism of order $d$ of the Dynkin diagram of $\Phi$ which
preserves the parameters $k$, and form the extended algebra
$\bH'=\bH\rtimes\<\delta\>.$  (The automorphism $\delta$ could of
course be trivial.) The cocenter $\bar \bH'=\bH'/[\bH',\bH']$ of $\bH'$ and the $\d$-twisted cocenter $\bar\bH_\delta=\bH/[\bH,\bH]_\delta$ of $\bH$ are related in section \ref{sec:3.1}. 

In section \ref{span-set setting}, we construct a set of elements
$\{w_C f_{J_C,i}\}$ of $\bH$, where $C$ runs over the $\delta$-twisted
conjugacy classes in $W$. To each class $C$, we attach  a
$\delta$-stable subset $J_C$ of the Dynkin diagram, and pick $w_C\in C\cap W_{J_C}$, where $W_{J_C}$ is the parabolic reflection subgroup of $W$ defined by $J_C$; the elements $f_{J_C,i}$ are chosen in $S(V)$, see \ref{span-set setting} for the precise definitions. Our
first result gives a basis for $\bar\bH_\delta$ (and hence a basis for
$\bar \bH'$), which is independent of the parameter function.
\begin{comment}
Let $\C I^\delta$ be the
set of all $\delta$-stable subsets of the Dynkin diagram, and choose a
set $\C I^\delta_0$ of representatives in $\C I^\delta$ under
appropriate conjugation by $W$, see section \ref{span-set setting} for the precise
definitions. For every
$\delta$-twisted conjugacy class $C$ in $W$, fix a minimal
element $J_C\in\C I^\delta_0$ such that $C$ meets the parabolic reflection
subgroup $W_{J_C}$, and fix an element $w_C$ in the intersection. 
Let $\{f_{J_C,i}\}$ be a
basis of $S(V^{W_{J_C}\rtimes\<\delta\>})^{N_{W,\delta}(W_{J_C})},$ where
$N_{W,\delta}(W_{J_C})$ is the $\delta$-twisted centralizer of $W_{J_C}.$ 
\end{comment}

\begin{ftheorem}\label{t:1}
The set $\{w_C f_{J_C,i}\}$ is a basis for the vector space $\bar\bH_\delta.$
\end{ftheorem}

The proof that the set $\{w_C f_{J_C,i}\}$ spans $\bar\bH_\delta$  relies
of certain results about $\delta$-twisted conjugacy classes in the
Weyl group, section \ref{sec:4}, as well as the use of a filtration in
$\bH$ and its associated graded object, which allows us to reduce the
proof to the case when the parameter function is identically $0$. The
case $k\equiv 0$ is proved directly in Proposition \ref{p:spank=0}.

For the linear independence we use the representation theory of $\bH$
to produce modules whose traces ``separate'' the elements $w_C
f_{J_C,i}.$ This is done in conjunction with a proof of the Density
theorem and (twisted) trace
Paley-Wiener theorem  for graded Hecke
algebras. More precisely, let $R^\delta(\bH)$ be the $\bZ$-span of the
$\delta$-stable irreducible $\bH$-modules $\Irr^\delta\bH$, and let
$R^*_\delta(\bH)=\Hom_\bC(R^\delta(\bH),\bC)$ be the (complex) dual
space. The twisted trace map is a linear map
$$\tr^\delta:\bar\bH_\delta\to R^*_\delta(\bH),$$
see section \ref{sec:4.1}. If $R^*_\delta(\bH)_\good$ is the subspace of good forms (Definition
\ref{d:goodforms}), the image of the trace map is automatically in
$R^*_\delta(\bH)_\good.$ 

\begin{ftheorem}\label{t:2}
The map $\tr^\delta:\bar\bH_\delta\to R^*_\delta(\bH)_\good$ is a linear isomorphism.
\end{ftheorem}

This is a graded affine Hecke algebra analogue of results from
$p$-adic groups, \cite{BDK},
\cite{Ka}, and \cite{Fl}. However, our proof of injectivity (which
uses the explicit spanning set of $\bar \bH_\delta$) is essentially different.
Our approach also leads to the following result on the dimension of the space
of $\delta$-elliptic representations $\overline R^\delta_0(\bH)$
(\ref{e:graded-Groth}). 

\begin{ftheorem}\label{t:3}
The dimension of the $\delta$-twisted elliptic representation space
$\overline R^\delta_0(\bH)$ is equal to the number of
$\delta$-twisted elliptic conjugacy classes in $W$.
\end{ftheorem}
When $\delta=1$ and the parameter function $k$ is positive, this
result was previously known from \cite{OS}, where it was obtained by
different methods.
Using the explicit description of the cocenter $\bar\bH_\delta,$ we
can argue that the dimension is at most the number of
$\delta$-elliptic conjugacy classes. To show equality, we construct
explicitly in section \ref{sec:9}, via a case-by-case analysis, a set
of linearly independent elements of $\overline R^\delta_0(\bH)$ with
the desired cardinality and other interesting properties, see Theorem \ref{t:explicit-basis}.

Finally, using Clifford theory for $\bH'$ and the relation between
$\bar\bH'$, $R(\bH')$ and $\bar\bH_{\delta^i}$, $R^{\delta^i}(\bH)$ ($i=1,d$)
respectively, we obtain:

\begin{fcorollary}
The trace map $\tr:\bar\bH'\to R(\bH')^*_\good$ is a linear isomorphism.
\end{fcorollary}

\medskip

For Hecke algebras with real parameters, similar results were announced
recently by Solleveld \cite{So2}, as part of his calculations of
Hochschild homology. The proofs are based on
deep results from \cite{So3}, where a version of the
Aubert-Baum-Plymen conjecture, involving the extended quotient of
the ``first kind'', is proved. In particular, the paper uses a $\bQ$-basis of $R(\bH')_\bQ$ which depends analytically on the real parameter function. 

\smallskip

Our method is different from {\it loc.~cit.} and appears to be more
elementary. For example, based on knowledge of minimal
length elements, (twisted) elliptic elements of finite Weyl groups (section
\ref{sec:4}), and explicit basis of the elliptic space $\overline R^\d_0(\bH)$, constructed in
section \ref{sec:9}, we are able to handle arbitrary complex
parameters. We also obtain explicit $\bQ$-basis of $R(\bH')_\bQ$
depending linearly on the complex parameter function. It is not clear
to us if there is a connection between the basis in \cite{So2} and
the one constructed in the present paper. Finally, our approach seems to be
related naturally to the Aubert-Baum-Plymen conjecture
with the extended quotient of the ``second kind'' \cite{ABP}.

\section{Preliminaries}\label{sec:1}

\subsection{Root system}\label{sec:1.1} Let $\Phi=(V_0,R,V^\vee_0,R^\vee)$
be a semisimple real root system: 
\begin{enumerate}
\item[(i)] $V_0$ and $V^\vee_0$ are finite dimensional real vector spaces with a
perfect pairing $(~,~):V_0\times V^\vee_0\to \bR$;
\item[(ii)] $R\subset V\setminus\{0\}$ spans $V$;
\item[(iii)] there is a bijection $R\leftrightarrow R^\vee$,
  $\al\leftrightarrow \al^\vee$, such that $(\al,\al^\vee)=2$;
\item[(iv)] the reflection $s_\al: V_0\to V_0$,
  $s_\al(v)=v-(v,\al^\vee)\al$ (resp., $s_\al: V_0^\vee\to V_0^\vee$,
  $s_\al(v^\vee)=v^\vee-(\al,v^\vee)\al^\vee$) preserves $R$
  (resp. $R^\vee$) for every $\al\in R$. 
\end{enumerate}
We assume, in addition, that the root system is reduced ($\al\in R$
implies $2\al\notin R$) and crystallographic ($(\al,\beta^\vee)\in
\bZ$ for all $\al,\beta\in R$).

Let $W$ be the finite Weyl group, i.e., the subgroup of $GL(V)$
(identified to a subgroup of $GL(V^\vee)$ too) generated by $s_\al,$
$\al\in R.$ Fix a choice of positive roots $R^+\subset R$ with
corresponding positive coroots $R^{\vee,+}\subset R^\vee.$ 
Denote $\rho=\frac 12\sum_{\al\in R^+}\al$ and $\rho^\vee=\frac
12\sum_{\al\in R^+}\al^\vee.$
Let
$I\subset R^+$ be a basis of $R^+$, the simple roots.

For every subset $J\subset I$, let $R_J$ be the subset of $R$
generated by $J$, $R_J^\vee$ the corresponding coroots, $R_J^+=R^+\cap
R_J,$ $R_J^{\vee,+}=R^{\vee,+}\cap R^\vee_J,$ and $W_J$ the parabolic Weyl
group generated by $\{s_\al: \al\in J\}$. Let
$\Phi_J=(V_0,R_J,V^\vee_0,R^\vee_J)$ be the (nonsemisimple) subroot
system. Denote
\begin{equation*}
V_{0,J}=\bR\text{-span of } R_J,\ V^\vee_{0,J}=\bR\text{-span of }
R^\vee_J,\quad \Phi_J^\sem=(V_{0,J},R_J,V^\vee_{0,J},R^\vee_J).
\end{equation*}

Set $V=\bC\otimes_\bR V_0$ and $V^\vee=\bC\otimes_\bR V^\vee_0,$ and
similarly for $V_J,$ $V^\vee_J.$  Let
$S(V)$ denote the symmetric algebra of $V$. Let $\bC[W]$ denote the
group algebra of $W$.

If $\delta$ is an automorphism of $I$ such that
$(\al,\beta^\vee)=(\delta(\al),\delta(\beta)^\vee)$, for all
$\al,\beta\in I$, then $\delta$
induces an automorphism of $\Phi$; denote $W'=W\rtimes
\langle\delta\rangle,$ where $\langle\delta\rangle$ is the cyclic
group generated by $\delta.$

\subsection{Graded affine Hecke algebra}\label{sec:1.2} Let $k:R^+\to
\bC$, $k(\al)=k_\al$, be a $W$-invariant function.

\begin{definition}[\cite{L1}]\label{d:graded} The graded affine Hecke algebra
  $\bH=\bH(\Phi,k)$ attached to the root system $\Phi$ and parameter
  function $k$ is the unique associative complex algebra with identity
  generated by $w\in W$ and $S(V)$ such that:
\begin{enumerate}
\item[(i)] $\bH\cong \bC[W]\otimes_\bC S(V)$ as
  $(\bC[W],S(V))$-bimodules;
\item[(ii)] $(1\otimes \omega)(s_\al\otimes 1)-(s_\al\otimes
  1)(1\otimes s_\al(\omega))=k_\al (\omega,\al^\vee),$ for all $\al\in
  I$, $\omega\in V.$
\end{enumerate}
In the sequel, we write $f$ for $1\otimes f$, $f\in S(V)$,  and $w$
for $w\otimes 1$, $w\in W.$ 
\end{definition}
From Definition \ref{d:graded}(ii), it is easy to deduce that
\begin{equation}
f\cdot s_\al-s_\al\cdot s_\al(f)=k_\al\Delta_\al(f),\ \al\in I,\ f\in S(V),
\end{equation} 
where $\Delta_\al(f)=\frac {f-s_\al(f)}\al$ is the difference
operator. Moreover, by induction on $w\in W$, one can then verify that
\begin{equation}
f\cdot w =w\cdot w^{-1}(f)+\sum_{w'<w} w' f_{w'},\ w\in W,\ f\in S(V),
\end{equation}
for some $f_{w'}\in S(V),$ where $<$ denotes the Bruhat order in
$W$. This relation will be used implicitly in the proofs below.

The center of $\bH$ is $Z(\bH)=S(V)^W$ (\cite[Proposition 4.5]{L1}). Since $\bH$ is
finite over $Z(\bH),$ every simple $\bH$-module is finite dimensional,
and the center $Z(\bH)$ acts by scalars (central character) in every irreducible
module. The central characters are thus parameterized by $W$-orbits in
$V^\vee.$ 
Denote $$\Theta(\bH)=W\backslash V^\vee\text{ and }\cc:\Irr\bH\to\Theta(\bH),$$
the central character map, a finite-to-one map. We say that an
irreducible $\bH$-module $\pi$ has real central character if
$\cc(\pi)\in W\backslash V^\vee_0.$

Let $\delta$ be an automorphism
  of $I$ as in section \ref{sec:1.1} and suppose the parameter
  function $k$ satisfies $k_\al=k_{\delta(\al)}$ for all $\al\in R^+.$
  In this case, $\delta$ defines an automorphism of $\bH$, and we may 
  define the extended graded affine Hecke algebra
  $\bH'=\bH\rtimes\langle\delta\rangle$. The center of $\bH'$ is
  $Z(\bH')=S(V)^{W'}$, and the central characters of $\bH'$-modules
  are parameterized by $W'$-orbits in $V^\vee.$ Denote
  $\Theta(\bH')=W'\backslash V^\vee.$

%\begin{remark}\label{r:order-delta} If $\delta$ is nontrivial, we will be interested in the situations when the root system $\Phi$ is
%irreducible and $\delta$ is of order less than or equal
%to $3$. (The case when $\delta$ is of order $3$ appears only if the
%root system is of type $D_4$.) In these cases, $k$ is necessarily a
%constant function.
%\end{remark}

Let $R(\bH')$ denote the Grothendieck group of finite dimensional
$\bH'$-modules. For every ring $K\supset\bZ,$ set
$R(\bH')_K=K\otimes_\bZ R(\bH').$

\subsection{The elements $\wti\om$}\label{sec:1.3}
The algebra $\bH'$ has a natural conjugate-linear anti-involution $*$
defined on generators (\cite[section 5]{BM}) by
\begin{equation}
w^*=w^{-1},\ w\in W,\quad \delta^*=\delta^{-1},\quad  \om^*=-\overline{w_0\cdot w_0(\om)\cdot w_0},\ \om\in V,
\end{equation}
where $w_0$ is the long Weyl group element. This definition is
motivated by the relation with Iwahori-Hecke algebras and $p$-adic
groups (see \cite{BM}). A direct computation shows that
\begin{equation}
\om^*=-\overline\om+2\overline p_\om, \text{ where }p_\om=\frac 12\sum_{\al\in R^+} k_\al (\om,\al^\vee) s_\al\in\bC[W],\quad \om\in V_0.
\end{equation}
Set
\begin{equation}\label{e:omtilde}
\wti\om=\om-p_\om, \quad \om\in V.
\end{equation}
In particular, $\wti\om^*=-\overline{\wti\om}$. 
\begin{comment}
The
elements $\wti\om$ satisfy the following well-known properties,
e.g. \cite{BM2} or \cite[Proposition 2.10]{BCT}:
\begin{enumerate}
\item $w\cdot\wti\om\cdot w^{-1}=\wti{w(\om)}$, $w\in W,$ $\om\in V$;
\item $[\wti\om_1,\wti\om_2]=-[p_{\om_1},p_{\om_2}],$ $\om_1,\om_2\in V.$
\end{enumerate}
\end{comment}
Notice also that $\delta(\wti\om)=\wti{\delta(\om)}.$

\subsection{A filtration of $\bH'$}\label{sec:1.4} Define a notion of
degree in $\bH'$ as follows. From Definition \ref{d:graded}, one sees
that every $h\in \bH'$ can be uniquely written as $h=\sum_{w\in W'} w
a_w,$ where $a_w\in S(V).$ Define the degree of $h$ to be the maximum
of degrees in $S(V)$ of all $a_w$. Set $\C F^j\bH'$ to be the set of
elements of $\bH'$ of degree less than or equal to $j$. This defines a filtration
\begin{equation*}
\bC[W']=\C F^0 \bH'\subset \C F^1\bH'\subset \C F^2\bH'\subset\dots,
\end{equation*}
and let $\tilde\bH'$ be the associated graded object. It is apparent from
the commutation relation in Definition \ref{d:graded} that $\tilde\bH'$
may be naturally identified with the (extended) graded affine Hecke
algebra $\bH'_0$ with parameter function $k=0$.

\subsection{Parabolic subalgebras}\label{sec:1.5}
Let $J\subset I$ be given. The parabolic subalgebra $\bH_J$ of $\bH$
is the subalgebra generated by $W_J$ and $S(V).$ Denote
$\bH_{J}^\sem$ the subalgebra of $\bH_J$ spanned by $W_J$ and $S(V_J)$;
this is the graded affine Hecke algebra for the semisimple root system
$\Phi_J^\sem$ with parameter function $k_J=k|_{R_J^+}.$ 
It is clear that
\begin{equation}
V=V_J\oplus V^{W_J},\quad \bH_J=\bH_J^\sem\otimes_\bC S(V^{W_J}).
\end{equation}
For every $\bH_J$ module $X$, define the parabolically induced
$\bH$-module
\begin{equation}
\Ind_{\bH_J}^\bH(X)=\bH\otimes_{\bH_J} X.
\end{equation}
In particular, if $\sigma$ is an $\bH_J^\sem$-module, and $\chi_\nu:
S(V^{W_J})\to \bC$ is a character parameterized by $\nu\in
(V^\vee)^{W_J}$, one can form the induced $\bH$-module
\begin{equation}
X(J,\sigma,\nu)=\Ind_{\bH_J}^\bH(\sigma\otimes \chi_\nu).
\end{equation}

\section{The cocenter and Clifford theory}\label{sec:2}

\subsection{$\delta$-commutators} \label{sec:3.1}
We retain the notation from the previous section. In particular,
$\delta$ is an automorphism of the Dynkin diagram of order $d$ and
$\bH'=\bH\rtimes\langle\delta\rangle$ is the extended graded affine
Hecke algebra.

\begin{definition}\label{d:deltacommut} Let $h, h' \in \bH$. Then $[h, h']_\d=h h'-h' \d(h)$ is called the {\it $\d$-commutator} of $h$ and $h'$. Let $[\bH, \bH]_\d$ be the submodule of $H$ generated by all $\d$-commutators. Similarly, we may define $[h, h']_{\d^i}$ and $[\bH, \bH]_{\d^i}$ for all $i \in \mathbf N$. 
\end{definition}

It is easy to see that $\d$ sends $[h, h']_{\d^i}$ to $[\d(h),
\d(h')]_{\d^i}$ and thus sends $[\bH, \bH]_{\d^i}$ to itself. Hence
$(1-\d): \bH \to \bH$ induces a map 
\begin{equation}
(1-\d): \bH/[\bH, \bH]_{\d^i} \to \bH/[\bH, \bH]_{\d^i}.
\end{equation}
 We denote by $\bar \bH^{[i]}$ the $\d$-coinvariants of $\bH/[\bH, \bH]_{\d^i}$, i.e., the quotient of $\bH/[\bH, \bH]_{\d^i}$ by the image of $(1-\d)$. 

We prove the following result. 

\begin{proposition}\label{cocenter-decomp}
Set $\bar \bH'=\bH'/[\bH', \bH']$, where
$\bH'=\bH\rtimes\<\d\>$. Then 

\begin{enumerate}
\item $\bar \bH'=\oplus_{i=0}^{d-1} \bH \d^i/([\bH', \bH'] \cap \bH \d^i)$, where $d$ is the order of $\d$. 

\item The map $h \mapsto h \d^i$ induces an isomorphism of vector spaces from $\bar \bH^{[i]}$ to $\bH \d^i/([\bH', \bH'] \cap \bH \d^i)$.
\end{enumerate}
\end{proposition}

\begin{proof}
We have the decompositions $\bH'=\oplus_{i=0}^{d-1} \bH \d^i$ and 
\begin{equation}
[\bH', \bH']=\sum_{j, k} [\bH \d^j, \bH \d^k]=\bigoplus_{i=0}^{d-1}
(\sum_j [\bH \d^{i+j}, \bH \d^{-j}]).
\end{equation}
 Notice that $\sum_j [\bH \d^{i+j}, \bH \d^{-j}] \subset \bH
 \d^i$. Thus $[\bH', \bH'] \cap \bH \d^i=\sum_j [\bH \d^{i+j}, \bH
 \d^{-j}]$ and $\bar \bH'=\oplus_{i=0}^{d-1} \bH \d^i/([\bH', \bH']
 \cap \bH \d^i)$.  Part (1) is proved. 

Now we show that 

\begin{enumerate}
\item[(a)] For any $i$, $\sum_j [\bH \d^{i+j}, \bH \d^{-j}] \d^{-i}=(1-\d) \bH+[\bH, \bH]_{\d^i}$. 
\end{enumerate}

By definition, $[\bH \d^{i+j}, \bH \d^{-j}]=[\bH \d^{i+j}, \d^{-j} \bH]$ is generated by $$[h \d^{i+j}, \d^{-j} h']=h \d^i h'-\d^{-j} h h' \d^{i+j}=\bigl(h \d^i(h')-\d^{-j}(h' h) \bigr) \d^i.$$ 

Notice that $$h \d^i(h')-\d^{-j}(h' h)=-[h', h]_{\d^i}+(1-\d^{-j})(h'
h) \in (1-\d) \bH+[\bH, \bH]_{\d^i},$$
since $(1-\d^{-j})(h' h) \in (1-\d) \bH$. 
 Thus $[\bH \d^{i+j}, \bH \d^{-j}] \subset (1-\d) \bH+[\bH, \bH]_{\d^i}$.

On the other hand, 
$$h-\d(h)=[h \d^{i-1}, \d]\delta^{-i} \in \sum_j [\bH \d^{i+j}, \bH \d^{-j}] \d^{-i}.$$
%$$h-\d(h)=h \d^i(1)-\d(1 \times h)=[h \d^{i-1}, \d] \in \sum_j [\bH
%\d^{i+j}, \bH \d^{-j}] \d^{-i}.$$ 
Hence $(1-\d) \bH \subset \sum_j [\bH \d^{i+j}, \bH \d^{-j}] \d^{-i}$. 

Moreover, as before, we have that $$[h', h]_{\d^i} \in -[h \d^{i+j}, \d^{-j} h'] \d^{-i}+(1-\d) \bH \subset \sum_j [\bH \d^{i+j}, \bH \d^{-j}] \d^{-i}.$$ 

Thus (a) is proved. 
Then \begin{align*} \bH \d^i/([\bH', \bH'] \cap \bH \d^i) & \cong \bH/([\bH', \bH'] \d^{-i} \cap \bH)=\bH/\sum_j [\bH \d^{i+j}, \bH \d^{-j}] \d^{-i} \\ &=\bH/\bigl((1-\d) \bH+[\bH, \bH]_{\d^i}\bigr). \end{align*}

It is easy to see that $(1-\d) \bH+[\bH, \bH]_{\d^i}/[\bH, \bH]_{\d^i}$ is the image of $(1-\d): \bH/[\bH, \bH]_{\d^i} \to \bH/[\bH, \bH]_{\d^i}$. Thus $\bH/\bigl((1-\d) \bH+[\bH, \bH]_{\d^i}\bigr)$ is isomorphic to the quotient of $\bH/[\bH, \bH]_{\d^i}$ by the image of $(1-\d)$. 

Part (2) is proved. 
\end{proof}

\subsection{Clifford theory for $\bH'$}\label{sec:3.2} 
Let $\Gamma=\<\d\>$. If $(\pi,X)$ is a finite dimensional $\bH$-module, let $(^{\delta^i} \pi,{}^{\delta^i}\!X)$
denote the $\bH$-module with the action
$^{\delta^i}\!\pi(h)x=\pi(\delta^{-i}(h))x,$ for all $x\in X,$ $h\in
\bH.$ Suppose $X$ is irreducible. Define the inertia group
$\Gamma_X=\{\delta^i: X\cong{}^{\delta^i}\!X\}.$ 

Fix a family of isomorphisms $\phi_{\delta^i}: X\to {}^{\delta^{-i}}\!
X$, $\delta^i\in \Gamma_X$ (each one of these isomorphisms is unique
up to scalar). In general, this defines  factor set ($2$-cocycle)
$$\beta:  \Gamma_X\times \Gamma_X\to \bC^\times,\quad
\phi_{\delta^i}\phi_{\delta^j}=\beta(\delta^i,\delta^j) \phi_{\delta^{i+j}}.$$
However, in our particular case, $\Gamma_X$ is a cyclic subgroup,
generated by say $\delta^{i_X}$ and we can normalize the isomorphisms
$\phi_{\delta^i}$ such that $\phi_{\delta^{k
    i_X}}=\phi_{\delta^{i_X}}^k.$ This has the consequence that the
factor set $\beta$ can be chosen to be trivial. 

If $U$ is an irreducible $\Gamma_X$-module, there is an action of
$\bH\rtimes \Gamma_X$ on $X\otimes U$: 
\begin{equation}
(h\delta^i)(x\otimes u)=h \phi_{\delta^i}(x)\otimes \delta^i u.
\end{equation}
One can form the induced $\bH'$-module $X\rtimes U=\Ind_{\bH\rtimes
  \Gamma_X}^{\bH\rtimes\Gamma} (X\otimes U).$ The main results in
Clifford theory are summarized next.

\begin{theorem}[{\cite[Appendix A]{RR}}]\label{t:irrH'} \begin{enumerate}
\item If $X$ is an irreducible $\bH$-module and $U$ an irreducible
  $\Gamma_X$-module, the induced $\bH'$-module $X\rtimes U$ is
  irreducible.
\item Every irreducible $\bH'$-module is isomorphic to an $X\rtimes U.$
\item If $X\rtimes U\cong X'\rtimes U'$, then $X,X'$ are
  $\<\d\>$-conjugate $\bH$-modules, and $U\cong U'$ as $\Gamma_X$-modules.
\end{enumerate}
\end{theorem}

We need a formula for the trace of an $\bH'$-module. For every
$\delta'\in\Gamma$ and $(\pi,X)\in\Irr^{\delta'}\bH$, let
$\phi_{\delta'}\in \End_\bC(X)$ be the intertwiner
as before. Define the twisted trace 
$$\tr^{\delta'}(\pi):\bH\to\bC,\
\tr^{\delta'}(\pi)(h)=\tr(\pi(h)\circ\phi_{\delta'}).$$ 
Let also $\Tr(~,~): \bH'\times R(\bH')\to \bC$ be the trace pairing,
i.e., $\Tr(h,\pi)=\tr\pi(h),$ $h\in\bH'$, $\pi\in R(\bH').$

\begin{lemma}\label{l:traceH'} Let $X\rtimes U$ be an irreducible
  $\bH'$-module as in Theorem \ref{t:irrH'}. For $h\in \bH,$ $\delta'\in\Gamma,$
\begin{equation}
\Tr(h\delta',X\rtimes U)=\begin{cases}
  \delta'(U)\sum_{\gamma\in\Gamma/\Gamma_X}\tr^{\delta'}(X)(\gamma^{-1}(h)),
  &\text{ if }\delta'\in\Gamma_X,\\
0, &\text{ if }\delta'\notin\Gamma_X,
\end{cases}
\end{equation}
where $\delta'(U)$ is the root of unity by which $\delta'$ acts in $U$.
\end{lemma}

\begin{proof}
As a vector space, $X\rtimes
U=\Ind_{\bH\rtimes\Gamma_X}^{\bH\rtimes\Gamma}(X\otimes
U)=\sum_{\gamma\in\Gamma/\Gamma_X}\gamma\otimes (X\otimes U).$
Then the action of $h\delta'$ is
\begin{align*}
(h\delta')\cdot \gamma\otimes (x\otimes u)&=h\delta'\gamma\otimes
(x\otimes u)=h\gamma'\delta''\otimes(x\otimes u),\text{ for some
}\gamma'\in \Gamma/\Gamma_X,\ \delta''\in \Gamma_X\\
&=\gamma'\cdot (\gamma')^{-1}(h)\delta''\otimes (x\otimes u)=\gamma'\otimes
((\gamma')^{-1}(h)\phi_{\delta''}(x)\otimes \delta''\cdot u).
\end{align*}
For this to have a nonzero contribution to the trace,
$\gamma'=\gamma,$ which is equivalent, since $\Gamma$ is abelian, with
$\delta'\in\Gamma_X.$ Suppose now this is the case (so
$\gamma'=\gamma$ and $\delta''=\delta'$). Then
$\delta'\cdot u=\delta'(U)u$ and the claim follows from the definition
of the twisted trace.

\end{proof}
   
{Set $R(\bH')_{\mathbb C}=R(\bH') \otimes_{\mathbb Z} \mathbb C$. We give a decomposition of  $R(\bH')_{\mathbb C}$} which is dual to the
decomposition $\bar\bH'=\oplus_{i=0}^{d-1}\bar\bH^{[i]}$ from
Proposition \ref{cocenter-decomp}, in a sense to be made
precise in the next section.

{Let $\co$ be a $\Gamma$-orbit on $\Irr(\bH)$. Set $\Gamma_\co=\Gamma_X$ for any $X \in \co$. This is well-defined since $\Gamma$ is cyclic. Then for any irreducible $\Gamma_\co$-module $U$ and $X \in \co$, $X \rtimes U=\oplus_{Y \in \co} Y \otimes U$ is independent of the choice of $X$. We denote it by $\co \rtimes U$. By Theorem \ref{t:irrH'}, $\Irr \bH'=\{\co \rtimes U\}$, where $\co$ runs over $\Gamma$-orbits on $\Irr(\bH)$ and $U$ runs over isomorphism classes of irreducible representations of $\Gamma_\co$. 

Suppose that $\d^i \in \Gamma_\co$. Let $U_{\co, i}$ be the virtual representation of $\Gamma_\co$ whose character is the characteristic function on $\d^i$. Then $\{\co \rtimes U_{\co, i}\}$ is a basis of $R(\bH')_\bC$. 

Let $R^{[i]}(\bH')_\bC$ be the subspace of $R(\bH')_\bC$ spanned by $\co \rtimes U_{\co, i}$, where $\co$ runs over $\Gamma$-orbits on $\Irr(\bH)$ with $\d^i \in \Gamma_\co$. Then 
\begin{equation}\label{groth-decomp}
R(\bH')_\bC=\oplus_{i=0}^{d-1} R^{[i]}(\bH')_\bC.
\end{equation}
By definition, $R^{[i]}(\bH')_\bC$ is a vector space with basis $(\Irr^{\d^i} \bH)_\Gamma$. Hence, the map $X \in \Irr^{\d^i}(\bH) \mapsto X \rtimes U_{\Gamma_X, i}$ induces an isomorphism $R^{\delta^i}(\bH)_{\Gamma, \mathbb C} \to R^{[i]}(\bH')_\bC$.  Here $R^{\delta^i}(\bH)_{\Gamma, \mathbb C}$ is the $\Gamma$-coinvariants of $R^{\delta^i}(\bH)_{\mathbb C}$.

By Lemma \ref{l:traceH'}, for $0 \le i, j<d$ with $i \neq j$, $\Tr(\bH \d^i, \co \rtimes U_{\co, j})=0$. }

%For every $0\le i<d$, let $\Irr^{\delta^i}\bH$ denote the set of irreducible $\bH$-modules $X$ fixed by $\delta^i$, i.e., $\delta^i\in \Gamma_X$. Let $R^{\delta^i}(\bH)_{\mathbb C}$ denote the $\bC$-span of $\Irr^{\delta^i}\bH$.

%Fix $\zeta$ a primitive $d$-th root of unity. For every $i$, we define an embedding 
%\begin{equation}
%R^{\delta^i}(\bH)\hookrightarrow R(\bH'),\quad X\in
%\Irr^{\delta^i}(\bH)\mapsto X\rtimes U_i,
%\end{equation}
%where $U_i$ is the $\Gamma_X=\langle\delta^{i_X}\rangle$-module on which $\delta^{i_X}$ acts by $(\zeta^{i_X })^i.$ Denote by $R^{[i]}(\bH)$ the image of $R^{\delta^i}(\bH)$ under this map. Then we have an abelian group decomposition
%\begin{equation}\label{groth-decomp}
%R(\bH')=\oplus_{i=0}^{d-1} R^{[i]}(\bH).
%\end{equation}

\section{(Twisted) Trace Paley-Wiener theorem}\label{sec:3}
In this section, we prove that trace Paley-Wiener theorem in the
setting of graded affine Hecke algebra. The proof follows the general
outline for the similar theorems for $p$-adic groups, \cite{BDK} and
\cite{Fl}, but for certain steps, e.g., Lemma \ref{l:finite-elliptic},
we give different arguments. 

\subsection{Trace forms}\label{sec:4.1} Define the trace linear map
\begin{equation}
\tr: \bH'\to R(\bH')^*,\quad h\mapsto (f_h: R(\bH')\to \bC,\
f_h(\pi)=\tr\pi(h)).
\end{equation}
It clearly descends to a linear map
\begin{equation}
\tr: \bar\bH'\to R(\bH')^*.
\end{equation}
This is compatible to the decompositions from Proposition
\ref{cocenter-decomp} and (\ref{groth-decomp}) as follows.
To simplify notation, we write the details in the case of $\delta$,
the same results hold for every $\delta^i.$ 
 Let
$R^*_{\delta}(\bH)=\Hom_\bC(R^{\delta}(\bH)_\bC,\bC)$ be the space of
$\bC$-valued linear forms on the vector space spanned by
$\Irr^{\delta}(\bH).$ 
The twisted trace map 
\begin{equation}
\tr_{\delta}: \bH\to R^*_{\delta}(\bH), \quad h\mapsto (f^{\delta}_h:
R^\delta(\bH)_\bC\to\bC, \ f^\delta_h(\pi)=\tr^\delta(\pi)).
\end{equation}
descends to a linear map
\begin{equation}
\tr^\delta:\bar\bH_\delta=\bH/[\bH,\bH]_\delta\to R^*_{\delta}(\bH).
\end{equation}
Call a form $f\in R^*_\delta(\bH)$ a trace form if $f=f^\delta_h$
for some $h\in\bH$ (or better $h\in \bar\bH_\delta$) and denote the
subspace of trace forms by $R^*_\delta(\bH)_{\tr}.$ This of course is
the image of $\tr^\delta.$

For every $J\subset I$ such that $\delta(J)=J,$ denote by
$(V^\vee)^{W_J\rtimes\delta}$ the fixed points of $\delta$ on
$(V^\vee)^{W_J}.$
\begin{definition}\label{d:goodforms}
A form $f\in R^*_\delta(\bH)$ is called good if for every $J\subset I$
such that $\delta(J)=J$, and every $\sigma\in\Irr^\delta(\bH_J^\sem)$,
the function $\nu\mapsto f(X(J,\sigma,\nu))$ is a regular function on
the variety $(V^\vee)^{W_J\rtimes\delta}$. Denote the subspace of good
forms by $R^*_\delta(\bH)_\good.$
\end{definition}
It is clear that $R^*_\delta(\bH)_{\tr}\subset R^*_\delta(\bH)_\good.$
The content of the trace Paley-Wiener theorem is that in fact the two spaces are equal:
\begin{theorem}\label{t:PW}
$R^*_\delta(\bH)_{\tr}= R^*_\delta(\bH)_\good.$
\end{theorem}
The proof is presented in the next subsections.

\subsection{A filtration of $R^\delta(\bH)$}\label{sec:4.2} 
For every $J\subset I$, recall the functor of parabolic induction
$i_J: R(\bH_J)\to R(\bH),$ $i_J(X)=\Ind_{\bH_J}^\bH(X)$ from section
\ref{sec:1.5}. For every $0\le \ell\le |I|,$ define the abelian subgroup
\begin{equation}
R^\delta_{\ell}(\bH)=\sum_{J=\delta(J)\subset I;~|J|\le |I|-\ell} i_J(R^\delta(\bH_J)).
\end{equation}
Then $R^\delta_0(\bH)=R^\delta(\bH)$, $R^\delta_1(\bH)=R^\delta_{\Ind}(\bH)=\sum_{J=\delta(J)\subset I}
i_J(R^\delta(\bH_J)),$ the subgroup of twisted parabolically induced
modules, and $R^\delta_{|I|}=0$. These subgroups form a decreasing filtration
$$
R^\delta(\bH)=R^\delta_0(\bH)\supset R^\delta_1(\bH)\supset R^\delta_2(\bH)\supset\dots.
$$
Set 
\begin{equation}\label{e:graded-Groth}
\overline
R^\delta_\ell(\bH)=R^\delta_\ell(\bH)/R^\delta_{\ell+1}(\bH),\quad 0\le\ell<|I|,
\end{equation}
and $\overline R^\delta(\bH)=\oplus_{0\le\ell<|I|} \overline
R^\delta_\ell(\bH),$ the associated graded group.

Of particular interest is $\overline
R^\delta_0(\bH)=R^\delta(\bH)/R^\delta_{\Ind}(\bH),$ the space of
virtual $\delta$-elliptic modules. A module $\pi\in \Irr^\delta\bH$ is
called elliptic if the image of $\pi$ in $\overline R^\delta_0(\bH)$
is nonzero. Let $\Theta^\delta(\bH)_0$ denote the set of elliptic
central characters, i.e., the subset of
$\Theta(\bH)$ of all central characters of elliptic $\pi\in\Irr^\delta\bH$.

\subsection{Langlands classification}\label{sec:4.3} The parabolic
induction part of the Langlands classification for $\bH$ is proved in
\cite{Ev},  see also \cite[Theorem 2.4]{KR}.

Let $(\pi,X)$ be a finite dimensional $\bH$-module. For every $\lambda\in
V^\vee$, set $$X_\lambda=\{x\in X: (\pi(a)-\lambda(a))^nx=0,\text{ for
  some }n\in\mathbb N, \text{ and all }a\in S(V)\},$$ the generalized
$\lambda$-weight space of $S(V).$  The set of $S(V)$-weights of $X$ is
$\Psi(X)=\{\lambda\in V^\vee: X_\lambda\neq 0\}.$ It is easy to see
that $\Psi(X)\subset W\cdot\nu$, where $\nu$ is (a representative of) the central character
of $X$.

For every $\nu\in V^\vee,$ write $\nu=\Re\nu+\sqrt{-1}\Im\nu$, where
$\Re\nu,\Im\nu\in V^\vee_0.$

\begin{definition}
An irreducible $\bH$-module $X$ is called tempered (resp., discrete
series) if for every $\lambda\in\Psi(X)$, $\langle
\om,\Re\lambda\rangle\le 0$ (resp. $\langle\om,\Re\lambda\rangle<0$) for all $\om\in V_0$ such that
$\langle\om,\al^\vee\rangle>0$ for all $\al\in I.$ 
\end{definition}

\begin{theorem}[{\cite[Theorem 2.1]{Ev}}]\label{t:Langlands}\ 
\begin{enumerate}
\item Every standard induced module $X(J,\sigma,\nu)$, where $\sigma$
  is a tempered $\bH_J^\sem$-module and 
  $\langle\al,\Re\nu\rangle>0$ for all $\al\in I\setminus J$, has a
  unique irreducible quotient $L(J,\sigma,\nu).$
\item Every irreducible $\bH$-module is isomorphic to a Langlands
  quotient $L(J,\sigma,\nu)$ as in (1).
\item If $L(J,\sigma,\nu)\cong L(J',\sigma',\nu'),$ then $J=J'$,
  $\sigma\cong \sigma'$, and $\nu=\nu'.$
\end{enumerate}
\end{theorem}

Moreover, it is implicitly proved in \cite{Ev} that if
$L(J',\sigma',\nu')$ is an irreducible constituent of the standard
module $X(J,\sigma,\nu)$ different than $L(J,\sigma,\nu)$, then
$\Re\nu>\Re\nu'$, where $>$ is the partial order relation on
$V_0^\vee$ defined by
\begin{equation}
a>b \text{ if } \langle \al,a\rangle>\langle \al,b\rangle,\text{ for
  all }\al\in I.
\end{equation}

It is immediate that $^\delta\!X(J,\sigma,\nu)\cong
X(\delta(J),{}^\delta\!\sigma,\delta(\nu)).$ Suppose that
$\pi\in\Irr^\delta\bH$ is given such that $\pi\cong
L(J,\sigma,\nu)$. Then $^\delta\!\pi\cong
L(\delta(J),{}^\delta\!\sigma,\delta(\nu))$ and Theorem
\ref{t:Langlands}(3) implies that necessarily $\delta(J)=J$,
$\sigma\in\Irr^\delta\bH^\sem_J$, and $\delta(\nu)=\delta$. In
particular, all irreducible constituents of $X(J,\sigma,\nu)$ are in $\Irr^\delta\bH.$

\begin{lemma}\label{l:temp-repr}
Suppose $\pi$ is an irreducible $\delta$-elliptic $\bH$-module. Then
there exists an irreducible $\delta$-elliptic tempered $\bH$-module
$\pi'$ such that $\cc(\pi)=\cc(\pi').$
\end{lemma}

\begin{proof}
This is the Hecke algebra analogue of \cite[Lemma 1.2]{Fl}. By Theorem
\ref{t:Langlands} and the remarks following it, $\pi\cong
L(J,\sigma,\nu)$ where $\delta(J)=J,$ $\sigma\in\Irr^\delta\bH^\sem_J$
and $\nu=\delta(\nu).$ In $R^\sigma(\bH)$, we have
$\pi=X(J,\sigma,\nu)-\sum_k\pi_k,$ where
$\pi_k=L(J_k,\sigma_k,\nu_k)\in\Irr^\delta\bH$ and $\nu_k<\nu.$ Thus,
by induction on the length of the $\nu$-parameter, it follows that
$\pi$ is a linear $\bZ$-combination of standard modules $$\pi=\sum_i a_i X(J_i',\sigma_i',\nu_i'),$$
where for every $i$, $\delta(J_i')=J_i'$,
$^\delta\!\sigma_i'\cong\sigma_i'$, $\delta(\nu_i')=\nu_i'$, and $\cc
X(J_i',\sigma_i',\nu_i')=\cc(\pi).$ If for every $i$, $J_i'\neq I$,
then $\pi\equiv 0$ mod $R^\sigma_{\Ind}(\bH)$, which is a
contradiction. Thus there exists $i$'s such that $J_i'=I,$ and the
corresponding $X(J_i',\sigma_i',\nu_i')=\sigma_i'$ are all tempered $\delta$-elliptic
$\bH$-modules (and have the same central character as $\pi$).
\end{proof}

\subsection{Induction and restriction in $R(\bH)$}\label{sec:4.4} If $K\subset
J(\subset I)$ are given, denote by $i_K^J: R(\bH_K)\to R(\bH_J)$ the
functor of induction, and by $r_K^J: R(\bH_J)\to R(\bH_K)$ the functor
of restriction. We also have the corresponding functors, denoted again
by $i_K^J$ and $r_K^J$ between $R^\delta(\bH_K)$ and $R^\delta(\bH_J).$
The following lemma is the analogue of \cite[Lemma 5.4]{BDK} and
\cite[Lemma 2.1]{Fl}.

\begin{lemma}\label{l:ind-restr}
\begin{enumerate}
\item[(i)] For $L\subset K\subset J\subset I,$ $i_L^J=i_K^J\circ
  i_L^K$ and $r_L^J=r_L^K\circ r_K^J.$
\item[(ii)] For $J,K\subset I$, let $^K W^J$ be a set of
  representatives of minimal length for $W_K\backslash W/W_J$. Then
\begin{equation}
r_K^I\circ i_J^I=\sum_{w\in {}^K W^J} i_{K_w}^K\circ w\circ r_{J_w}^J,
\end{equation}
where $K_w=K\cap wJw^{-1}$ and $J_w=J\cap w^{-1}K w.$
\item[(iii)] If $K=wJw^{-1}$, then $i_K^I\circ w=i_J^I.$
\item[(iv)] If $K=\delta(K)$ and $J=\delta(J),$ let
  $^K W^J(\delta)\subset {}^K W^J$
  denote the subset of $\delta$-fixed elements. Then
\begin{equation}
r_K^I\circ i_J^I=\sum_{w\in {}^K W^J(\delta)} i_{K_w}^K\circ w\circ r_{J_w}^J,
\end{equation}
as functors from $R^\delta(\bH_J)$ to $R^\delta(\bH_K).$
\end{enumerate}
\end{lemma}

\begin{proof}
Claim (i) is obvious in our setting. 

We prove claim (ii). We need to prove that
\begin{equation}
\bH\otimes_{\bH_J}\sigma=\bigoplus_{w\in{}^KW^J}\bH_K\otimes_{\bH_{K_w}}
w\circ \sigma|_{\bH_{J_w}},\text{ as left }\bH_K\text{-modules,}
\end{equation}
where $K_w=K\cap wJ w^{-1}$ and $J_w=w^{-1}K_w w.$

Let $\{w_1,w_2,\dots,w_t\}$ be the set of
elements in ${}^KW^J$ ordered such that $\ell(w_i)\le \ell(w_{i+1}).$
As  $(\bH_K,\bH_J)$-bimodules, 
$$\bH=\sum_{i=1}^t \bH_Kw_i\bH_J.$$
Define a $(\bH_K,\bH_J)$-bimodule filtration of $\bH$ by setting
$\bH_s=\sum_{i\le s}\bH_Kw_i\bH_J.$ Then
$E_s=\bH_s\otimes_{\bH_J}\sigma$ defines a left $\bH_K$-module filtration of
$\bH\otimes_{\bH_J}\sigma$. Set
$\overline E_s=E_s/E_{s-1},$ a left $\bH_K$-module.
Then 
\begin{equation}\label{e:ind-assoc-graded}
\bH\otimes_{\bH_J}\sigma\cong\oplus_{s}\overline E_s,
\end{equation}
as left
$\bH_K$-modules, so one needs to prove that there exists an $\bH_K$-module isomorphism
\begin{equation}
B_s: \overline E_s\to \bH_K\otimes_{\bH_{K_{w_s}}} w_s\circ (\sigma|_{\bH_{J_{w_s}}}).
\end{equation}
Notice that $\overline E_s$ is generated by
$w_s\otimes_{\bH_J}\sigma$ as a $\bH_K$-module. Set 
\begin{equation}
B_s(hw_s\otimes v) = h\otimes_{\bH_{K_{w_s}}} \tau_{w_s}(v), \quad
h\in\bH_K,\ v\in\sigma,
\end{equation}
where $\tau_{w_s}$ is the isomorphism $\sigma\to w_s\circ\sigma.$ We
need to check that $B_s$ is well-defined. Since $\bH_{K_{w_s}}$ is
generated by $W_{K_{w_s}}$ and $S(V)$, it is sufficient to check on
these generators. 

Firstly, if $w\in W_{K_{w_s}}$, $w w_s=w_s w',$ where $w'\in
W_{J_{w_s}}$. We have $B_s(w w_s\otimes_{\bH_J}v)=w
\otimes_{\bH_{K_{w_s}}} \tau_{w_s}(v)=1\otimes_{\bH_{K_{w_s}}}
\sigma^{w_s}(w)\tau_{w_s}(v)= 1\otimes_{\bH_{K_{w_s}}} \tau_{w_s}(\sigma(w')v).$ On the
other hand, $B_s(w_s w'\otimes_{\bH_J}
v)=B_s(w_s\otimes_{\bH_J}
\sigma|_{J_{w_s}}(w')v)=1\otimes_{\bH_{K_{w_s}}} \tau_{w_s}(\sigma(w')v).$

Secondly, let $a\in S(V)$ be given. Then $$a w_s=w_s\cdot
w_s^{-1}(a)+\sum_{u<w_s} u a_u,$$
for some $u\in W$, $a_u\in S(V).$ This means that $a w_s\equiv w_s\cdot w_s^{-1}(a)$
in $\overline E_s$, i.e., modulo $E_{s-1}$. In the same way as for
$w\in W_{K_{w_s}}$, it is then easy to see that $B_s(a w_s)=B_s(w_s\cdot
w_s^{-1}(a)).$

Thus $B_s$ is well-defined, and it is clearly a surjective
$\bH_K$-homomorphism. Thus $\dim \overline E_s\ge \dim
\bH_K\otimes_{\bH_{K_{w_s}}} w_s\circ
(\sigma|_{\bH_{J_{w_s}}})=|W_K/W_{K_{w_s}}| \dim\sigma.$ Summing over
$s$ and using (\ref{e:ind-assoc-graded}), we find
$|W/W_J|\dim\sigma=\sum\dim \overline E_s\ge \sum_s |W_K/W_{K_{w_s}}|
\dim\sigma.$ But since $|W/W_J|=\sum_s |W_K/W_{K_{w_s}}|,$ it follows
that every $B_s$ must in fact be an isomorphism. Claim (ii) is proved.

Claim (iii) follows from (ii) and the parabolic induction part
of Langlands classification (Theorem \ref{t:Langlands}) identically with
the proof of Lemma 5.4 of \cite{BDK}. For (iv), one can adapt the
proof of (ii) exactly as in \cite[Lemma 2.1(iv)]{Fl}.

\end{proof}

For every $J=\delta(J)\subset I$, define the operator
\begin{equation}
T_J: R^\delta(\bH)\to R^\delta(\bH), \quad T_J=i_J^I\circ r_J^I.
\end{equation}
Formal manipulations with the properties in Lemma \ref{l:ind-restr}
yield the following formulas (see \cite[Corollary 5.4]{BDK}).
\begin{lemma}\label{l:Toper}
\begin{enumerate}
\item[(i)] $T_K\circ i_J^I=\displaystyle{\sum_{w\in {}^K W^J(\delta)}} i_{J_w}^I\circ
  r^J_{J_w},$ where $J_w=J\cap w^{-1} K w.$
\item[(ii)] $T_K\circ T_J=\displaystyle{\sum_{w\in {}^K W^J(\delta)}} T_{J_w}.$
\end{enumerate}
\end{lemma}
As a consequence, one sees that the operators $T_K$ respect the
filtration $\{R^\delta_\ell(\bH)$ from section
\ref{sec:4.2}. Moreover, if $\ell=|K|$, then $T_K$ acts on the
quotient $\overline R^\delta_\ell(\bH)$ by:
\begin{equation}
T_K i_J^I(\sigma)=\begin{cases}|W_K^\delta| i_J^I(\sigma),&\text{ if
  }J \sim K,\\
0,&\text{ if }|J|=\ell\text{ and }J\not\sim K.\end{cases}
\end{equation}
As in \cite[section 5.5]{BDK}, define 
\begin{equation}
A_\ell=\prod_{K=\delta(K), |K|=\ell}(T_K-|W_K^\delta|),\text{ and }
A=A_{|I|}\circ A_{|I|-1}\circ\dots\circ A_0.
\end{equation}
Since every $A_\ell$ preserves the filtration and kills $\overline
R^\delta_\ell(\bH),$ it follows that $A$ kills $R^\sigma_{\Ind}(\bH).$
\begin{definition}
A linear form $f\in R^*_\delta(\bH)$ is called $\delta$-elliptic if
$f(R^\delta_{\Ind}(\bH))=0.$ Let $R^*_\delta(\bH)_0$ denote the space
of $\delta$-elliptic linear forms.
\end{definition}
Thus, the adjoint operator $A^*: R^*_\delta(\bH)\to R^*_\delta(\bH)$
has the image in $R^*_\delta(\bH)_0$. On the other hand, from Lemma
\ref{l:Toper}, it is apparent that $A$ is of the form
$A=a+\sum_{J=\delta(J)\subsetneq I} c_J' T_J$, for some integers
$a\neq 0$ and
$c_J'.$ But then $A^*=a+\sum_{J=\delta(J)\subsetneq I} c_J' T_J^*,$
where $T_J^*=r_J^*\circ i_J^*.$ (Here for simplicity, we write $r_J^*$
instead of $(r_J^I)^*$ and similarly for $i_J^*$.) Set $\C A=\frac 1a
A^*.$ 

\subsection{Inclusion and restriction for $\bH$}\label{sec:4.5}
For every $J\subset I$, let $\wti i_J: \bH_J\to \bH$
denote the inclusion. Define $\wti r_J:\bH\to \bH_J$ as follows. Given
$h\in \bH$, let $\psi_h:\bH\to\bH$ be the linear map given by left
multiplication by $h.$ This can be viewed as a right $\bH_J$-module
morphism. Since $\bH$ is free of finite rank right $\bH_J$-module, with
basis $W^J$, one can consider $\tr\psi_h\in\bH_J.$ Set $\wti
r_J(h)=\tr\psi_h.$ Set $\wti T_J=\wti i_J\circ\wti r_J:\bH\to \bH.$
As in section \ref{sec:4.3}, for every $K\subset J,$ we may also
define $\wti i_K^J$ and $\wti r_K^J.$

\begin{lemma}\label{l:adjointtrace}
The maps $\wti i_J$ and $\wti r_J$ are $\Tr(~,~)$-adjoint to $r_J$ and
$i_J$, respectively,i.e.:
\begin{enumerate}
\item[(a)]$\Tr(\wti i_J(h),\pi)=\Tr(h,r_J(\pi)),$ for $h\in\bH_J,$ $\pi\in
  R(\bH)$;
\item[(b)]$\Tr(h,i_J(\pi))=\Tr(\wti r_J(h),\pi),$ for $h\in\bH$, $\pi\in R(\bH_J)$.
\end{enumerate}
Thus $\wti T_J$ is $\Tr(~,~)$-adjoint to $T_J$ as well.
\end{lemma}

\begin{proof}
Claim (a) is obvious. For (b), let $M$ be the space of the
representation $\pi$ and $\{v_1,\dots,v_n\}$ an orthonormal basis of
$M$ with
respect to an inner product $\<~,~\>_M$. A basis for $i_J(\pi)$ is
$\{x\otimes v_i: i=1,n,\ x\in W^J\}.$ Define $\<~,~\>$ on $i_J(\pi)$
by declaring this basis orthonormal. For $h\in \bH$, view left
multiplication of $h$ as a right $\bH_J$-module map $\bH\to \bH$, and
then $$h\cdot x=x
h_x+\sum_{x'\in W^J\setminus\{x\}} x' h_{x'},$$
for some $h_x,h_{x'}\in\bH_J.$ Notice that $\wti r_J(h)=\sum_{x\in
  W^J} h_x.$ Then
\begin{align*}
\Tr(h,i_J(\pi))&=\sum_{i=1}^n\sum_{x\in W^J}\<h\cdot x\otimes
v_i,v_i\>=\sum_{i=1}^n \sum_{x\in W^J}\<x h_x\otimes v_i,v_i\>\\
&=\sum_{i=1}^n
\sum_{x\in W^J}\<x\otimes \pi(h_x)v_i,x\otimes
v_i\>=\sum_{i=1}\sum_{x\in W^J}\<\pi(h_x)v_i,v_i\>_M=\Tr(\wti
r_J(h),\pi).
\end{align*}
\end{proof}

The analogous discussion with section \ref{sec:4.4} holds
here and also the $\delta$-twisted version. In particular, define the
filtration of $\bH$:
$$\mathcal E^0_\delta\bH\supset \mathcal E^1_\delta\bH\supset\dots\supset\mathcal
E^\ell_\delta\bH\supset\dots,$$
where $\mathcal E^\ell_\delta\bH=\sum_{J\subset I, \delta(J)=J, |J|\le
  |I|-\ell} \wti i_J(\bH_J)=\sum_{J\subset I, \delta(J)=J, |J|\le
  |I|-\ell} \bH_J.$ Set $\overline{\mathcal E\bH}^\ell_\delta=\mathcal
E^\ell_\delta\bH/\mathcal E^{\ell+1}_\delta\bH$ and 
$$\wti A_\ell=\prod_{\delta(K)=K,|K|=\ell} (\wti
T_K-|W_K^\delta|),\quad \wti A=\wti A_{|I|}\circ A_{|I|-1}\circ\dots.$$
As before, $\wti A_\ell$ preserves the filtration and kills $\mathcal
E^{\ell+1}_\delta\bH.$ Thus
\begin{equation}\label{e:Atilde}
\wti A(\sum_{J\subset I, \delta(J)=J}\bH_J)=0.
\end{equation}

\subsection{Proof of Theorem \ref{t:PW}}\label{sec:4.6} 
Denote $\Irr^\delta(\bH)_\el=\cc^{-1}(\Theta^\delta(\bH)_0)$, and
recall that $\cc$ is
finite-to-one. Set $R^\delta(\bH)_\el$ to be the linear span of $\Irr^\delta(\bH)_\el$.

\begin{lemma}\label{l:finite-elliptic}
The space $\overline R^\delta_0(\bH)$ is finite-dimensional, in particular,
the set $\Irr^\delta(\bH)_\el$ is finite.
\end{lemma}

\begin{proof}
Suppose $\{\pi_1,\pi_2,\dots,\pi_k\}$ is a set in $\subset
  R^\delta(\bH)$ such that its image in $\overline R^\delta_0(\bH)$ is
  linearly independent. Applying the operator $A$ from section
  \ref{sec:4.3}, one obtains a linearly independent set
  $\{A(\pi_1),A(\pi_2),\dots,A(\pi_k)\}$ in $R^\delta(\bH).$ This is
  because $A(\pi)\equiv a \pi$ in $\overline R^\delta_0(\bH)$, for a
  nonzero integer $a$.

  Since the characters of simple modules are linear independent, so
  are the characters of any linear independent set in $R^\delta(\bH).$
  Thus there exist elements $h_1,h_2,\dots,h_k$ of $\bH$, such
  that the matrix $(\Tr(h_i\delta,A(\pi_j))_{i,j}$ is invertible. By Lemma
  \ref{l:adjointtrace}, the matrix $(\Tr(\wti A(h_i)\delta,\pi_j)_{i,j}$ is
  invertible. Since $\wti A$ vanishes on $\sum_{J,\delta(J)=J}\bH_J$ by
  (\ref{e:Atilde}), it follows that
$$k\le \dim \bH/\left([\bH,\bH]_\delta+\mathcal E^1_\delta\bH\right).$$
By Proposition \ref{p:span} proved in section \ref{sec:6}, the right hand side is bounded above by
the number of $\delta$-elliptic conjugacy classes in $W$. This proves
the first claim.

For the second claim, for every central character $\lambda\in
\Theta^\delta(\bH)_0$, let $R^\delta(\bH)_\lambda$ be the span of
$\cc^{-1}(\lambda)\subset \Irr^\delta\bH$, and let $\overline
R^\delta_0(\bH)_\lambda$ be the image of $R^\delta(\bH)_\lambda$ in
$\overline R^\delta_0(\bH).$ Then 
$$R^\delta(\bH)_\el=\bigoplus_{\lambda\in \Theta^\delta(\bH)_0}
R^\delta(\bH)_\lambda\text{ and }\overline R^\delta_0(\bH)=\bigoplus_{\lambda\in \Theta^\delta(\bH)_0}
\overline R^\delta_0(\bH)_\lambda.$$
This is because irreducible $\bH$-modules with different central
characters are necessarily independent, and the central character is
the same for all constituents of a parabolically induced from a module
with central character.

Since $\overline R^\delta_0(\bH)$ is finite dimensional, then
$\Theta^\delta(\bH)_0$ must be finite. Since $\cc$ is a finite to one
map, $\Irr^\delta(\bH)_\el$ is also finite.
\end{proof}

\begin{remark} The proof of Lemma \ref{l:finite-elliptic} we presented
  is different than the argument from \cite{BDK}. The classical proof
  (adapted to this setting under the assumption that $k$ is real valued)
  shows that the set of $\delta$-elliptic central characters
  $\Theta^\delta(\bH)_0$ is finite, as follows.
 Firstly, the set $\Theta^\delta(\bH)_0$ is a finite union
of locally closed (in the Zarisky topology) subsets of $\Theta$, see
\cite[Proposition 1.1]{Fl}. Secondly, let $*: \Theta(\bH)\to
\Theta(\bH)$ be the anti-algebraic involution given by the hermitian dual. More
precisely, if $\nu\in \Theta(\bH)$ is the central character of an
irreducible module $\pi$, let $\nu^*$ be the central character of the
hermitian dual of $\pi$ with respect to the operation $*$ from section
\ref{sec:1.3}. Since $k$ is real, it follows from \cite[Proposition 2.35]{Op} that every tempered $\bH$-module is
$*$-unitary. In particular, using Lemma \ref{l:temp-repr}, $\nu=\nu^*$ for every 
 $\nu\in \Theta^\delta(\bH)_0.$ It follows that $\Theta^\delta(\bH)_0$
 is finite.
\end{remark}

Let $f\in R^*_\delta(\bH)_\good$ be given. Since $\Irr^\delta(\bH)_\el$ is
a finite set and the (twisted) characters of irreducible $\bH$-modules are
linearly independent, we can
choose $f_1\in R^*_\delta(\bH)_{\tr}$ such that $f(\pi)=f_1(\pi)$ for
all $\pi\in \Irr^\delta(\bH)_\el.$ By replacing $f$ with $f-f_1$, we
may therefore assume, without loss of generality, that
$f(R^\delta(\bH)_\el)=0.$

Apply to $f$ the operator $\C A$ defined in the previous section. Then
$\C A(f)\in R^*_\delta(\bH)_0$, i.e., 
$\C A(f)(R^\delta_{\Ind}(\bH))=0.$ Recall that $\C
A=1+\sum_{J=\delta(J)\subsetneq I} c_J T_J^*,$ for some
$c_J\in\mathbb Q.$  The operators $T_J$ preserve central characters,
and therefore preserve $R^\delta(\bH)_\el$, hence $\C A(f)$ vanishes
on $R^\delta(\bH)_\el$, because $f$ does. Since
$R^\delta(\bH)=R^\delta(\bH)_\el+R^\delta_{\Ind}(\bH)$, it follows
that $\C A(f)=0$, and thus $f=-\sum_{J=\delta(J)\subsetneq I} c_J
T_J^*(f).$

By induction on $|J|$, we may assume that $
R^*_\delta(\bH_J^\sem)_\good= R^*_\delta(\bH_J^\sem)_{\tr}.$ Since
$\bH_J=\bH_J^\sem\otimes S(V^{W_J}),$ it is straightforward that also
$R^*_\delta(\bH_J)_\good= R^*_\delta(\bH_J)_{\tr}.$ It is also easy to
see that
\begin{equation}
i_J^*(R^*_\delta(\bH)_\good)\subset R^*_\delta(\bH_J)_\good\text{ and }r_J^*(R^*_\delta(\bH_J)_{\tr})\subset R^*_\delta(\bH)_{\tr};
\end{equation}
in the case of $p$-adic groups, the second inclusion requires an
argument, see \cite[section 5.3]{BDK}, but since for $\bH$, $r_J$ is
just restriction, it is immediate.

Thus $T_J^*(f)=r_J^*(i_J^*(f))$ is in $R^*_\delta(\bH)_{\tr},$ and so
is $f$, concluding the proof.

\

{Now we define the good forms for $\bH'=\bH \rtimes \Gamma$. 

\begin{definition}
For any $J \subset I$ and $\sigma\in\Irr(\bH_J^\sem)$, we set $$\Gamma_{J, \s}=\{\d^i; \d^i(J)=J, \d^i \s \cong \s\}.$$ 

A form $f\in R^*(\bH')$ is called good if for every $J\subset I$, $\sigma\in\Irr(\bH_J^\sem)$ and irreducible representation $U$ of $\Gamma_{J, \s}$, 
the function $\nu\mapsto f(X(J,\sigma,\nu) \rtimes U)$ is a regular function on
the variety $(V^\vee)^{W_J\rtimes\Gamma_{J, \s}}$. Denote the subspace of good
forms by $R^*(\bH')_\good.$
\end{definition}

\

As a consequence of Theorem \ref{t:PW} and Clifford theory (section \ref{sec:3.2}), we obtain the trace Paley-Wiener Theorem for $\bH'$. 

\begin{corollary}
$R^*(\bH')_{\tr}=R^*(\bH')_{\good}$. 
\end{corollary}

\begin{proof}
It is obvious that $R^*(\bH')_{\tr} \subset R^*(\bH')_{\good}$. By (\ref{groth-decomp}), $$R(\bH')_\bC=\oplus_{i=0}^{d-1} R^{[i]} (\bH')_\bC \cong \oplus_{i=0}^{d-1} R^{\d^i}(\bH)_{\Gamma, \bC}.$$ Hence $$R^*(\bH')=\oplus _{i=0}^{d-1} \Hom_\bC(R^{\d^i}(\bH)_{\Gamma, \bC}, \bC)=\oplus_{i=0}^{d-1} \Hom_\bC(R^{\d^i}(\bH)_\bC, \bC)^\Gamma=\oplus_{i=0}^{d-1} R^*_{\d^i}(\bH)^\Gamma.$$ 

By definition, $R^*(\bH')_{\good} \subset \oplus_{i=0}^{d-1} R^*_{\d^i}(\bH)^\Gamma_{\good}$, where $R^*_{\d^i}(\bH)^\Gamma_{\good}=R^*_{\d^i}(\bH)^\Gamma \cap R^*_{\d^i}(\bH)_{\good}$. Notice that for $0 \le i, j<d$, $\tr(\bH \d^i) \mid_{R^{[j]}(\bH')_\bC}=0$ unless $i=j$. 

By Theorem \ref{t:PW}, the image of the map $\tr_{\d^i}: \bH/[\bH, \bH]_{\d^i} \to R^*_{\d^i}(\bH)$ is $R^*_{\d^i}(\bH)_{\good}$. Hence the image of $\bar \bH^{[i]}$ is $R^*_{\d^i}(\bH)^\Gamma_{\good}$. By Proposition \ref{cocenter-decomp}, $$\tr(\bH')=\oplus_{i=0}^{d-1} R^*_{\d^i}(\bH)^\Gamma_{\good}=R^*(\bH')_{\good}.$$ 
\end{proof}
}

\section{Twisted elliptic conjugacy classes in the finite Weyl
  group}\label{sec:4}

In this section, we discuss the (twisted) conjugacy classes of finite Coxeter groups. These results will be used in the rest of this paper. In this section, we fix a finite irreducible Coxeter group $(W, I)$ and a group automorphism $\d: W \rightarrow W$ with $\d(I)=I$. Let $d$ be the minimal positive integer such that $\d^d(i)=i$ for all $i \in I$.  

\subsection{Twisted conjugacy classes} For $w \in W$, set $\supp_{\d}(w)=\cup_{n=0}^{d-1} \d^n \supp(w)$. Then $\supp_{\d}(w)$ is a $\d$-stable subset of $I$. 

We define the {\it $\d$-twisted conjugation action} of $W$ on itself by $w \cdot_\d w'=w w' \d(w) \i$. Any orbit is called a {\it $\d$-twisted conjugacy class} of $W$. A $\d$-conjugacy class $\co$ of $W$ is called {\it elliptic} if $\co \cap W_J=\varnothing$ for all proper $\d$-stable subset $J$ of $I$, i.e., $\supp_\d(w)=I$ for all $w \in \co$. An element $w \in W$ is called {\it $\d$-elliptic} if it is contained in an elliptic $\d$-conjugacy class of $W$. 

Recall that $V$ is the vector space spanned by $\a_i$ (for $i \in I$). As
before, we regard $W$ as a subgroup of $GL(V)$ and $\d$ as an element in $GL(V)$ in the natural way. For $w \in W$, set $$p_{w, \d}(q)=\det(q \cdot id_V-w \d).$$ Then it is easy to see that $p_{w, \d}(q)=p_{w', \d}(q)$ if $w$ is $\d$-conjugate to $w'$.

We have the following well-known result for elliptic conjugacy classes. We include the proof here for completeness. 

\begin{proposition}
Let $\co$ be a $\d$-twisted conjugacy class of $W$. The following are equivalent:

(1) $\co$ is elliptic; 

(2) $p_{w, \d}(1) \neq 0$ for some (or equivalently, any) $w \in \co$;

(3) For some (or equivalently, any) $w \in \co$, there is no nonzero point in $V$ that is fixed by $w \d$. 
\end{proposition}

\begin{proof}
(2)$\Rightarrow$(1) is proved in \cite[Lemma 7.2]{He07}. 

(3)$\Rightarrow$(2) is obvious. 

(1)$\Rightarrow$(3): Let $w \in \co$ and $v \neq 0$ with $w \d(v)=v$. Let $x \in W$ with $x(v)$ dominant. Set $\bar v=x(v)$ and $w'=x \i w \d(x)$. Then $w' \in \co$ and $w' \d(\bar v)=\bar v$. Since $\bar v$ is dominant, $\d(\bar v)$ is also dominant. Hence $\d(\bar v)-\bar v=\d(\bar v)-w' \d(\bar v)$ is a linear combination of $\a_i$ with nonnegative coefficients. As $(\d(\bar v), \rho^\vee)=(\bar v, \rho^\vee)$, we must have that $\d(\bar v)=\bar v$ and $w' \d(\bar v)=\d(\bar v)$. Hence $w'$ is generated by $s_\a$, where $\a$ runs over simple roots in $V$ such that $(\bar v, \a^\vee)=0$. In particular, $w'$ is in a proper $\d$-stable parabolic subgroup of $W$. Thus $\co$ is not elliptic. 
\end{proof}

\subsection{Minimal length elements} We follow the notation in \cite[section 3.2]{GP}.

Given $w, w' \in W$ and $i \in I$, we write $w
\xrightarrow{s_i}_{\d} w'$ if $w'=s_i w \d(s_i)$ and $\ell(w') \le
\ell(w)$. If $w=w_0, w_1, \cdots, w_n=w'$ is a sequence of elements in
$W$ such that for all $k$, we have $w_{k-1} \xrightarrow{s_j}_{\d}
w_k$ for some $j \in I$, then we write $w \rightarrow_{\d} w'$. If $w \rightarrow_{\d} w'$ and $w' \rightarrow_{\d} w$, then we say that $w$ and $w'$ are in the same $\d$-cyclic shift class and write $w \approx_{\d} w'$. 

For $w \in W$ and $i \in I$, define the length function $\ell_i(w)$ as the number of generators in $I$ conjugate to $s_i$ occurring in a reduced expression of $w$. By \cite[Exercise 1.15]{GP}, it is independent of the choice of reduced expression of $w$. 

Set $$l_{i, \d}(w)=\sum_{k=0}^{d-1} l_{\d^k(i)} (w).$$ Then it is easy to see that if $w \approx_{\d} w'$, then $l_{i, \d}(w)=l_{i, \d}(w')$ for all $i \in I$.

\

We have the following main result on elliptic conjugacy classes of $W$. 

\begin{theorem}\label{min} Let $\co$ be an elliptic $\d$-twisted conjugacy class in $W$ and $\co_{\min}$ be the set of minimal length elements in $\co$. Then

(1) For each $w \in \co$, there exists $w' \in \co_{\min}$ such that
$w \rightarrow_{\d} w'$.

(2) Let $w, w' \in \co_{\min}$, then $w \approx_{\d} w'$. In particular, $l_{i, \d}(w)=l_{i, \d}(w')$ for all $i \in I$.

(3) Let $\co'$ be an elliptic $\d$-conjugacy classes of $W$. Let $w \in \co_{\min}$ and $w' \in \co'_{\min}$. Then $\co=\co'$ if and only if $p_{w, \d}(q)=p_{w', \d}(q)$ and $l_{i, \d}(w)=l_{i, \d}(w')$ for all $i \in I$.
\end{theorem}

\begin{remark} It was first prove via a case-by-case analysis for untwisted case by Geck and Pfeiffer in \cite[Theorem 3.2.7]{GP} and for twisted case by the second-named author in \cite[Theorem 7.5]{He07}. A case-free proof for part (1) and (2) was found recently in \cite{HN}. It would be interesting to find a case-free proof for part (3) and/or Theorem \ref{neverfuse} below. 
\end{remark}

\

The following result can be checked easily from the list of Dynkin diagrams. 

\begin{lemma}\label{wjd}
Let $J \subset I$ with $\d(J)=J$. Then we may write $J$ as $J=J_1 \sqcup J_2$ with $\d(J_i)=J_i$ for $i=1, 2$ and 

(1) $J_1$ is a union of connected components of type $A$;

(2) For any connected component $K$ of $J_1$, either $\d \mid_K$ is identity or there exists another connected component $K' \neq K$ of $J_1$ with $\d(K)=K'$;

(3) Either (i) $J_2=\emptyset$ or (ii) $J_2$ is a connected component of $J$ not of type $A$ or (iii) $J_2$ is a connected component of $J$ of type $A$ and $\d \mid_{J_2}$ is nontrivial. 
\end{lemma}

\

We have the following consequence that elliptic classes never fuse. 

\begin{theorem}\label{neverfuse}
Let $J \subset I$ with $\d(J)=J$. Let $C$ be a $\d$-twisted conjugacy class of $W$ such that $C \cap W_J$ contains a $\d$-elliptic element of $W_J$. Then $C \cap W_J$ is a single $\d$-twisted conjugacy class of $W_J$. 
\end{theorem}

\begin{remark}
The untwisted case was due to Geck and Pfeiffer in \cite[Theorem 3.2.11]{GP}. The general case can be proved in a similar way by using Theorem \ref{min} (3) and Lemma \ref{wjd}. We omit the details. 
\end{remark}

\

In the rest of this section, we discuss some further properties on elliptic conjugacy classes of a parabolic subgroup of $W$. 

\begin{proposition}\label{j-j'}
Let $C$ be a $\d$-twisted conjugacy class of $W$. Let $J$ be a minimal $\d$-stable subset of $I$ with $C \cap W_J \neq \emptyset$. Let $w \in C \cap W_J$ and $J' \subset I$ with $\d(J')=J'$. Then $x w \d(x) \i \in W_{J'}$ if and only if $x=x' x_1$ for some $x' \in W_{J'}$ and $x_1 \in W^{\d} \cap {}^{J'} W^J$ such that $x_1(J) \subset J'$. 
\end{proposition}

\begin{remark}
It is easy to see that $J$ is a minimal $\d$-stable subset of $I$ with $C \cap W_J \neq \emptyset$ if and only if $C \cap W_J$ contains a $\d$-elliptic element of $W_J$. 
\end{remark}

\begin{proof}
If $x=x' x_1$ for some $x' \in W_{J'}$ and $x_1 \in W^{\d} \cap {}^{J'} W^J$ such that $x(J) \subset J'$, then $x w \d(x) \i=x' (x_1 w x_1 \i) \d(x') \i \in W_{J'}$. 

Suppose that $x w \d(x) \i \in W_{J'}$. We write $x$ as $x=y' x_1 y$, where $y \in W_J$, $y' \in W_{J'}$ and $x_1 \in {}^{J'} W^J$. Set $w'=y w \d(y) \i \in C\cap W_J$. Then $x_1 w' \d(x_1) \i \in W_{J'}$. Hence $x_1 w' \in x_1 W_J \subset W_{J'} x_1 W_J$ and $x_1 w' \in W_{J'} \d(x_1) \subset W_{J'} \d(x_1) W_J$. Since $x_1, \d(x_1) \in {}^{J'} W^J$, we must have that $x_1=\d(x_1) \in W^{\d}$. 

Now $w' \in W_J \cap x_1 \i W_{J'} x_1=W_{J \cap x_1 \i(J')}$. By our assumption on $J$, $J \cap x_1 \i(J')=J$. So $x_1(J) \subset J'$. Hence $x \in W_{J'} x_1$. 
\end{proof}

\begin{proposition}\label{center-w}
Let $J \subset I$ with $\d(J)=J$ and $w$ be a $\d$-elliptic element in $W_J$. Set $Z_{W, \d}(w)=\{x \in W; x w \d(x) \i=w\}$ and $N_{W, \d}(W_J)=\{x \in W; x W_J \d(x) \i=W_J\}$. Then $$W_J Z_{W, \d}(w)=N_{W, \d}(W_J)=W_J Z W_J,$$ where $Z=\{z \in {}^J W^J; z=\d(z), J=z(J)\}$.  
\end{proposition}

\begin{proof}Let $C$ be the $\d$-twisted conjugacy class of $W$ that contains $w$ and $C_J$ be the $\d$-twisted conjugacy class of $W_J$ that contains $w$. Let $x \in N_{W, \d}(W_J)$. By Theorem \ref{neverfuse}, $x w \d(x) \i \in C \cap W_J=C_J$. Therefore $x w \d(x) \i=y w \d(y) \i$ for some $y \in W_J$. Hence $y \i x \in Z_{W, \d}(w)$ and $x \in W_J Z_{W, \d}(w)$. 

On the other hand, if $x \in W_J Z_{W, \d}(w)$, then $x w \d(x) \i \in
C_J$. By Proposition \ref{j-j'}, $x \in W_J Z W_J \subset N_{W,
  \d}(W_J)$. 

\end{proof}

%We write $x$ as $x=y_1 z y_2$, where $y_1, y_2 \in W_J$ and $z \in {}^J W^J$. Then $(z y_2) w (z y_2) \i \in z C_J \d(z) \i \cap C_J$. In particular, $z W_J \d(z) \i \cap W_J \neq \emptyset$ and $W_J z W_J \cap W_J \d(z) W_J \neq \emptyset$. Since $z \in {}^J W^J$, $\d(z) \in {}^J W^J$. Thus $z=\d(z)$. So $$\emptyset \neq z C_J \d(z) \i \cap C_J \subset C_J \cap (W_J \cap z W_J z \i)=C_J \cap W_{J \cap z(J)}.$$ Notice that $C_J$ is an elliptic conjugacy class of $W_J$. Thus $J \cap z(J)=J$ and $z \in Z$. So $x \in W_J Z W_J \subset N_{W, \d}(W_J)$. \qed

\section{Spanning set of (twisted) cocenter}\label{sec:6}

\subsection{Spanning set of $\bH_0/[\bH_0, \bH_0]_\d$}\label{span-set setting} Let $\mathcal I^\d=\{J \subset I; J=\d(J)\}$. For $J, J' \in \mathcal I^\d$, we write $J \sim_\d J'$ if there exists $w \in W^\d \cap {}^{J'} W^J$ such that $w(J)=J'$. For each $\sim_\d$-equivalence class in $\mathcal I^\d$, we choose a representative. Such representatives form a subset of $\mathcal I^\d$, which we denote by $\mathcal I^\d_0$. Then there is a natural bijection between $\mathcal I^\d_0$ and $\mathcal I^\d/\sim_\d$. 

For any $J \subset \mathcal I^\d_0$, we choose a basis $\{f_{J, i}\}$ of $S(V^{W_J \rtimes \<\d\>})^{N_{W, \d}(W_J)}$. 

For each $\d$-twisted conjugacy class $C$ of $W$, we fix a minimal element $J_C \in \mathcal I^\d_0$ such that $C \cap W_{J_C} \neq \emptyset$. Such $J_C$ is uniquely determined by $C$.  We fix an element $w_C$ in $C \cap W_{J_C}$. By definition, $w_C$ is a $\d$-elliptic element in $W_{J_C}$. 

\begin{proposition}\label{p:spank=0} We keep the notations as in $\S$\ref{span-set setting}. Then $\{w_C f_{J_C, i}\}$ spans $\bH_0/[\bH_0, \bH_0]_\d$ as a vector space. 
\end{proposition}

\begin{proof}
Notice that for any $x, y \in W$ and $f \in S(V)$, $$x y \d(x) \i f \equiv y \d(x) \i f \d(x)=y \d(x)\i(f) \mod [\bH_0, \bH_0]_\d.$$ 

Thus $\bar \bH_0$ is spanned by $w_C S(V)$, where $C$ runs over $\d$-twisted conjugacy classes of $W$. 

Now we fix a $\d$-twisted conjugacy class $C$. Let $J=J_C$. We have that $V=V^{W_J} \oplus U$, where $U$ is the subspace of $V$ spanned by simple roots in $J$. This is a decomposition of $V$ as $W_J \rtimes \<\d\>$ submodules. Since $w$ is $\d$-elliptic in $W_J$, $(1-w_C \d)$ is invertible on $U$. Also $V^{W_J}=V^{W_J \rtimes \<\d\>} \oplus (1-\d) V^{W^J}$. As $(1-\d) V^{W_J}$ is a direct sum of eigenspaces of $\d$ with eigenvalues not equal to $1$, $(1-w_C \d)$ is also invertible on $(1-\d) V^{W^J}$. Set $U'=U \oplus (1-\d) V^{W^J}$. Then $(1-w_C \d)$ is invertible on $U'$ and $V^{w_C \d}=V^{W_J \rtimes \<\d\>}$. 

Let $f \in S(V)$ and $u \in U'$. There exists $v \in U'$ such that $v-w_C \d(v)=u$. Now \begin{align*} u w_C f &=v w_c f-\bigl(w_C \d(v) \bigr) w_C f=v w_c f-w_C f \d(v) \\ &=[v, w_C f]_\d \in [\bH_0, \bH_0]_\d. \end{align*}

Hence $U' w_C S(V) \in [\bH_0, \bH_0]_\d$. 

Since $V=V^{W_J \rtimes \<\d\>} \oplus U'$, we have that \begin{align*} w_C S(V) &=w_C S(U') S(V^{W_J \rtimes \<\d\>})=S(U') w_C S(V^{W_J \rtimes \<\d\>}) \\ & \subset w_C S(V^{W_J \rtimes \<\d\>})+[\bH_0, \bH_0]_\d.\end{align*}

Let $f \in S(V^{W_J \rtimes \<\d\>})$. Then for any $x \in Z_{W, \d}(w_C)$, $$w_C f \equiv x w_C f \d(x) \i=(x w_C \d(x) \i) [\d(x)(f)]=w_C \d(x)(f) \mod [\bH_0, \bH_0]_\d.$$ Hence $$w_C f \equiv \frac{1}{|Z_{W, \d}(w_C)|} \sum_{x \in Z_{W, \d}(w_C)} w_C \d(x)(f) \in w_C S(V^{W_J \rtimes \<\d\>})^{Z_{W, \d}(\d(w_C))} \mod [\bH_0, \bH_0]_\d.$$ Here the last inclusion follows from the fact that $\d(Z_{W, \d}(w_C))=Z_{W, \d}(\d(w_C))$. 

Since $w_C$ is $\d$-elliptic in $W_J$, $\d(w_C)$ is also $\d$-elliptic in $W_J$. By Proposition \ref{center-w}, $W_J Z_{W, \d}(\d(w_C))=N_{W, \d}(W_J)$ and $$S(V^{W_J \rtimes \<\d\>})^{Z_{W, \d}(\d(w_C))}=S(V^{W_J \rtimes \<\d\>})^{W_J Z_{W, \d}(\d(w_C))}=S(V^{W_J \rtimes \<\d\>})^{N_{W, \d}(W_J)}.$$ 

Therefore $\{w_C f_{J, i}\}$ spans the image of $w_C S(V)$ in $\bH_0/[\bH_0, \bH_0]_\d$. 
\end{proof}

\subsection{Spanning set of $\bar \bH_\d$} Let $\tilde \bH$ be the graded algebra associated to the
filtration of $\bH$ given by the degree of $S(V)$ defined in section
\ref{sec:1.4}. Recall that $\tilde \bH
\cong \bH_0$. By induction on degree one shows that if
$\mathcal L$ is a spanning set of $\tilde \bH/[\tilde \bH, \tilde
\bH]_{\d^i}$, then $\mathcal L$ is also a spanning set of $\bH/[\bH,
\bH]_{\d^i}$. To see this, first notice that (using the relations in
$\bH$) the $\delta^i$-commutators preserve the filtration, i.e., $$[\C
F^l\bH,\C F^j\bH]\subset \C F^{l+j}\bH.$$
More precisely, if $a_1$ and $a_2$ are elements of $S(V)$ homogeneous
of degrees $l$ and $j$, respectively, then 
\begin{equation}\label{e:commute}
[w_1\cdot a_1,w_2\cdot a_2]_{\delta^i}\in (w_1w_2\cdot
w_2^{-1}(a_1)a_2-w_2\delta^{i}(w_1) \cdot
\delta^i(a_1)\delta^{i}(w_1^{-1})(a_2))+\C F^{l+j-1}\bH.
\end{equation}
Let $\bar h\in \bH/[\bH,\bH]_{\delta^i}$ be
given and assume that $h\in \C F^l\bH.$ Since $h$ can be written as $h=\sum_w w a_w$ for
some $a_w\in S(V),$ set $h_0=\sum_w w a_w\in \bH_0$. Then
$h_0-\sum_{\bar x\in \C L} c_x x_0\in [\bH_0,\bH_0]_{\delta^i}$, where
only finitely many $c_x$ are nonzero, $x_0$ is a representative of $\bar
x$ in $\bH_0$. Moreover, we may choose the $x_0$'s that contribute to
the sum to have (maximal) degree less than or equal to $i$. Thus
$h_0=\sum_{\bar x} c _x x_0+\sum_j[y_{0,j},y'_{0,j}]$, for some
$y_{0,j},y'_{0,j}$ in $\bH_0$ homogeneous and $[y_{0,j},y'_{0,j}]\in
\C F^l\bH_0$. 

Let $y_j,y_j'\in \bH$ be the corresponding elements for $y_{0,j},y'_{0,j}$,
respectively.  By (\ref{e:commute}), $[y_j,y_j']_{\delta^i}$ differs
from 
$[y_{0,j},y'_{0,j}]_{\delta^i}$ by an element in $\C F^{l-1}\bH$. Then $h-\sum_{\bar x} c _x
x-\sum_j[y_{j},y'_{j}]\in \C F^{l-1}\bH,$ and the claim follows by
induction. (Here $x$ denotes a representative of $\bar x$ in $\bH$
corresponding to $x_0$.)

Now we have the following result. 

\begin{proposition}\label{p:span}
We keep the notations as in $\S$\ref{span-set setting}. Then $\{w_C f_{J_C, i}\}$ spans $\bar \bH_\d$ as a vector space. 
\end{proposition}

\

Now we give the trace formula of the element from the spanning set on the induced representations. 

\begin{proposition}\label{trace-formula}
Let $J, J'$ be $\d$-stable subsets of $I$. Let $w$ be an $\d$-elliptic element in $W_J$ and $C$ be the $\d$-twisted conjugacy class of $W$ that contains $w$. Let $M$ be an $\bH'_{J'}$-module. Then for any $f \in S(V^{W_J \rtimes \<\d\>})^{N_{W, \d}(W_J)}$, \[\Tr(w f \d, \Ind^{\bH'}_{\bH'_{J'}} M)=\begin{cases} |N_{W, \d}(W_J)/W_J| \Tr(w f \d, M), & \text{ if } J=J'; \\ 0, & \text{ if } C \cap W_{J'}=\emptyset. \end{cases}\]
\end{proposition}

\begin{proof}
The space $\Ind^{\bH'}_{\bH'_{J'}} M$ is the direct sum of $x M$, $x \in W^{J'}$. 

Let $x \in W^{J'}$ and $m \in M$. We write $x$ as $x=x_1 x_2$, where $x_1 \in W_J$ and $x_2 \in {}^J W^{J'}$. Then \[w f \d(x m)=w f\d(x_1) \d(x_2) m=w \d(x_1) f \d(x_2) m.\] Notice that $f \d(x_2)=\d(x_2) [\d(x_2) \i(f)]+\sum_{y<\d(x_2)} y f_y=\d(x_2) [\d(x) \i(f)]+\sum_{y<\d(x_2)} y f_y$ for some $f_y \in S(V)$. We have that \[w f \d(x m)=w \d(x) [\d(x) \i(f)] m+\sum_{y<\d(x_2)} w \d(x_1) y f_y m.\]

We show that 

(a) For any $y<\d(x_2)$, $w \d(x_1) y f_y m \in \sum_{x' \in W^{J'}, x' \neq x} x' M$. 

If (a) is not true, then $w \d(x_1) y \in x W_{J'}$. Hence $y \in (w \d(x_1)) \i x W_{J'} \subset W_J x_2 W_{J'}$. Since $x_2 \in {}^J W^{J'}$, $y \ge x_2$. Thus $\ell(y) \ge \ell(x_2)=\ell(\d(x_2))>\ell(y)$. That is a contradiction. Hence (a) is proved. 

It is easy to see that $w \d(x) [\d(x) \i(f)] m \in \sum_{x' \in W^{J'}, x' \neq x} x' M$ if and only if $w \d(x) \in x W_{J'}$, i.e., $x \i w \d(x) \in W_{J'}$. 

In particular, if $C \cap W_{J'}=\emptyset$, then $x \i w \d(x) \notin W_{J'}$ for any $x$ and thus $\Tr(w f \d, \Ind^{\bH'}_{\bH'_{J'}} M)=0$. 

If $J'=J$, then by Proposition \ref{center-w}, for any $x \in W^J$, $x \i w \d(x) \in W_J$ if and only if $x \in {}^J W^J \cap W^\d$ and $x(J)=J$. For such $x$, $\d(x) \i(f)=f$. The number of elements $x$ satisfy such condition equals $|N_{W, \d}(W_J)/W_J|$. 

Now \[\Tr(w f \d, \Ind^{\bH'}_{\bH'_J} M)=\sum_{x \in {}^J W^J \cap W^\d, x(J)=J} \Tr(x \i w f \d x, M).\] We have that $x \i w f \d x=x \i w x  f \d$. By Theorem \ref{neverfuse}, there exists $x' \in W_J$ such that $x \i w x=x' w \d(x') \i$. Then $x \i w f \d x=x' w \d(x') \i f \d=x' (w f \d) (x') \i$. So $\Tr(x \i w f \d x, M)=\Tr(x' w f \d (x') \i, M)=\Tr(w f \d, M)$. Thus $$\Tr(w f \d, \Ind^{\bH'}_{\bH'_J} M)=|N_{W, \d}(W_J)/W_J| \Tr(w f \d, M).$$ 
\end{proof}

\section{Density theorem and basis theorem}\label{sec:7}

\subsection{Tempered modules} Let $\res_{W'}: R(\bH')\to R(W')$ denote the
linear map given by the restriction of $\bH'$-modules to $\bC[W'].$
Define
\begin{equation}
\C T_\rcc=\{\pi: \pi\text{ tempered $\bH'$-module with real central
  character}\}.
\end{equation}

\begin{theorem}\label{t:W-ind} Suppose the parameter function $k$ of $\bH$ is real-valued.
\begin{enumerate}
\item[(a)]The set $\res_{W'}\C T_\rcc$ is a $\bQ$-basis of $R(W')_\bQ.$ 

\item[(b)]If $\bH'$ has parameter
  function of geometric type in the sense of \cite{L2}, then $\res_{W'}\C
  T_\rcc$ is in fact a  $\bZ$-basis of $R(W')$ and the change of bases
  matrix between $\res_{W'}\C T_\rcc$ and $\Irr W'$ is upper
  uni-triangular in an appropriate ordering. 
\end{enumerate}
\end{theorem}

Part (a) of Theorem \ref{t:W-ind} is proved by homological algebra in
\cite[Theorem 6.5]{So}. Part (b) follows from
\cite{L2} together with Clifford theory, see \cite[3.4-3.6]{BC}. In fact, (b) is now known to hold for
all real parameters $k$ and all root systems except some cases in type
$F_4$ (where it is again expected to be true).

\subsection{Basis Theorem}\label{sec:7.2}
We retain the notation from section \ref{span-set setting}. In particular, let
$\{w_C f_{J_C,i}\}$ be the spanning set of $\bar\bH_\delta$ from
Proposition \ref{p:span}, indexed by conjugacy classes $C$ in $W$ and
for every $C$, a basis $\{f_{J_C,i}\}$ of $S(V^{W_{J_C}\rtimes\<\delta\>})^{N_{W,\delta}(W_{J_C})}.$

\begin{theorem}[Basis Theorem]\label{t:basis}
The set $\{w_C f_{J_C,i}\}$ forms a basis of $\bar\bH_\delta.$
\end{theorem}

\begin{proof}
In light of Proposition \ref{p:span} we need to prove that the set is
linearly independent. 
To see this, we proceed by induction and use the formula for the trace
of induced modules in Proposition \ref{trace-formula} to separate the
subsets for various $J'.$
Let 
\begin{equation}\label{e:lin-comb}
\sum_{C,i} a_{C,i} w_C f_{J_C,i}=0
\end{equation}
 be a linear combination.

The base case is $J'=\emptyset$. For every character $\chi_\nu:
\bH_\emptyset=S(V)\rtimes \<\delta\>\to \bC,$ parameterized by $\nu\in (V^\vee)^\delta$, consider the induced
module $X(\nu):=\Ind_{\bH_\emptyset}^\bH(\chi_\nu)$. By Proposition
\ref{trace-formula}, $\Tr(w_C f_{C,i}\delta,X(\nu))=0$ for all
$C\neq\{1\}$ and all $i$. Thus applying $\Tr(-,X(\nu))$ to
(\ref{e:lin-comb}), we get $|W|\sum_{i} a_{1,i}
(f_{\emptyset,i},\nu)=0$ for every $\nu\in W\backslash (V^\vee)^\delta$. This
means that the polynomial function $\sum_i a_{1,i} f_{\emptyset,i}$
vanishes on its natural domain, thus $\sum_i a_{1,i}
f_{\emptyset,i}=0$ in $S(V^{\delta})^{W}$. Since by
construction, $\{f_{\emptyset,i}\}$ are linear independent in
$S(V)^{W\rtimes \<\delta\>}$, it follows that $a_{1,i}=0$ for all $i.$ 

By induction, suppose we are left with a (smaller) linear combination
as in (\ref{e:lin-comb}). Let $J'$ be a minimal element in $\mathcal I^\d_0$ that appears
in this combination. Suppose $C$ is a conjugacy class that occurs and
$C\cap W_{J'}\neq \emptyset.$ By the construction of the spanning set
(see section \ref{span-set setting}) and the minimality of $J',$ we
must have $J'=J_C.$ Let $M$ be a module of $\bH'_{J'}$. Applying
$\Tr(-,\Ind_{\bH'_{J'}}^{\bH'}M)$ to the linear combination, it follows
that
$$|N_{W,\delta}(W_{J'})/W_{J'}|\sum_{C: J_C=J',i} a_{C,i} \Tr(w_C
f_{J',i}\delta,M)=0, \text{ for all }M.$$
Specialize $M$ to $M=\sigma\otimes \chi_\nu$, where $\sigma$ is an irreducible
tempered module with real central character of
$(\bH_{J'})^\sem$, and $\chi_\nu:S(V^{W_J\rtimes \<\delta\>})\to \bC$
is a character indexed by $\nu\in
N_{W,\delta}(W_J)\backslash(V^\vee)^{W_J\rtimes\<\delta\>}.$ By
Theorem \ref{t:W-ind}, when the parameter function $k$ takes real
values, the set of such $\sigma$ separates the $w_C$'s. For arbitrary
parameters $k$, we can specialize $\sigma$ to the representations
explicitly constructed in Theorem \ref{t:explicit-basis}(1) below.
By the same discussion as for $J'=\emptyset$, the unramified
characters $\chi_\nu$ separate the $f_{J',i}$'s. In conclusion,
$a_{C,i}=0$ for all $C$ such that $J_C=J'.$
\end{proof}

\begin{remark}
Theorem \ref{t:basis} implies a description of the cocenter for the extended graded
Hecke algebra $\bH'=\bH\rtimes\<\delta\>$, $\delta^d=1$. From Proposition \ref{cocenter-decomp}, we have 
\begin{equation}
\bar \bH'=\bigoplus_{i=0}^{d-1} \bar\bH^{[i]},\text{ where
}\bH^{[i]} \text{ is the space of }\delta\text{-coinvariants in }\bar \bH_{\delta^i}.
\end{equation}
Now Theorem \ref{t:basis} gives an explicit basis for each $\bar\bH_{\delta^i}.$
\end{remark}

The following corollary is stated here for convenience. It was already
proved in Lemma \ref{l:finite-elliptic} and it is a consequence of the
explicit description of the basis of $\overline \bH_\delta$. The case $\delta=1$ was
proved in the setting of affine Hecke algebras in \cite[Proposition 3.9]{OS} via
different methods.

\begin{corollary}\label{c:dimension-elliptic}
The dimension of the $\delta$-elliptic space $\overline
R^\delta_0(\bH)$ is at most the number of $\delta$-twisted elliptic
conjugacy classes in $W$.
\end{corollary}

\subsection{Density Theorem}\label{sec:7.3} In the notation of section
\ref{sec:4.1},  let $\tr^\delta:\bar\bH_\delta\to R^*_\delta(\bH)$ be
the trace map. 

\begin{theorem}[Density Theorem]\label{t:density} The trace map
  $\tr^\delta:\bar\bH_\delta\to R^*_\delta(\bH)$ is
  injective. Equivalently, if $h\in\bH$ is
  such that $\Tr(h\delta,\pi)=0$ for all $\pi\in\Irr^\delta(\bH)$, then $h\in [\bH,\bH]_\delta.$
\end{theorem}

\begin{proof}
Suppose $x=\sum_{C,i} a_{C,i} w_C f_{J_C,i}\in \bar\bH_\delta$ is in
$\ker\tr^\delta.$ This means that for every $\pi\in R^\delta(\bH)_\bQ$,
$\Tr(x\delta,\pi)=0$. The inductive argument in the
proof of Theorem \ref{t:basis} then implies that $a_{C,i}=0$ for all
$C,i,$ and so $x=0.$ 
\end{proof}

In the setting of $\bH'=\bH\rtimes\<\delta\>$, we defined the trace
map $\tr:\bar\bH'\to R(\bH')^*$ in section \ref{sec:4.1}. As a
consequence of Theorem \ref{t:density} and Clifford theory (section
\ref{sec:3.2}), we obtain a density theorem for $\bH'$. 

\begin{corollary} The trace map $\tr:\bar\bH'\to R(\bH')^*$ is
  injective. 
\end{corollary}

\begin{proof}
{By Proposition \ref{cocenter-decomp}, $\bar \bH'=\oplus_{i=0}^{d-1} \bH \d^i/([\bH', \bH'] \cap \bH \d^i)$. By (\ref{groth-decomp}), $R(\bH')_\bC=\oplus_{i=0}^{d-1} R^{[i]}(\bH')_\bC$. Moreover, for $0 \le i, j<d$, $\tr(\bH \d^i) \mid_{R^{[i]}(\bH')_\bC}=0$ unless $i=j$. Thus $\ker \tr=\oplus_{i=0}^{d-1} \ker \tr ~\cap ~\bH \d^i$. }

Let $h\delta'\in\bH'$, $\delta'=\delta^i,$ be such that
$\tr(h\delta')=0.$ By Lemma \ref{l:traceH'}, this is equivalent with
the condition that for every $X\in\Irr^{\delta'}\bH,$ 
$$\sum_{\gamma\in\Gamma/\Gamma_X}\tr^{\delta'}(X)(\gamma^{-1}(h))=0.$$
Since for every $\gamma_1\in\Gamma_X,$
$\tr(X)(h)=\tr({}^{\gamma_1}X)(h)=\tr(X)(\gamma_1^{-1}(h))$, it follows
that the equivalent condition is
$$\sum_{\gamma\in\Gamma}\tr^{\delta'}(X)(\gamma^{-1}(h))=0.$$
By Theorem \ref{t:density}, we have then that
$\sum_{\gamma\in\Gamma}\gamma(h)\in [\bH,\bH]_{\delta'}.$
 Writing \[d h=\sum_{\gamma}\gamma(h)+\sum_{\gamma\neq 1} (1-\gamma)(h),\]
we see that $h\in (1-\delta)\bH+[\bH,\bH]_{\delta'}.$ By claim (a) in the proof
of Proposition \ref{cocenter-decomp}, this is the same as $h\delta'\in
\bH\delta'\cap [\bH',\bH']$, which is what we wanted to prove.

\end{proof}

\section{Bases of $R(\bH)$}\label{sec:9}

In this section, we exhibit linearly independent sets in $\overline
R^\delta_0(\bH)$ of cardinality equal to the number of
$\delta$-elliptic conjugacy classes in $W$. This implies in particular
that equality holds in Corollary
\ref{c:dimension-elliptic}, and that our sets are bases. Moreover, our
construction is such that the $W$-structure of these bases elements
does not change for various values of the parameter function $k$ of
$\bH$, and in addition the action of the Hecke algebra elements
$\om$ (equivalently $\wti\om$) depends linearly in $k.$

\subsection{}The main result of the section follows.

\begin{theorem}\label{t:explicit-basis} Let $\bH_k=\bH(\Psi,k)$ be a
  graded Hecke algebra associated to a simple root system $\Psi$
  and parameter function $k:R^+\to\bC$ as in Definition
  \ref{d:graded}. Let $\delta$ be an automorphism of the Dynkin
  diagram of $\Psi$ of order less than or equal to $3$.
\begin{enumerate}
\item There exists a $\bZ$-basis of $\overline R^\delta_0(\bH_k)$
  represented by a set of genuine representations
  $\{\pi_1,\dots,\pi_m\}\subset R^\delta(\bH_k)$,
  where $m$ is the number of $\delta$-elliptic conjugacy classes in
  $W$ such that
\begin{enumerate}
\item[(P1)] the restriction $\res_W\pi_j$ is independent of $k$ for every
  $j=1,m$;
\item[(P2)] the actions of $\pi_j(\om)$ $j=1,m$, $\om\in V_\bC$, depend
  linearly in the parameter function $k$. 
\end{enumerate}
\item There exists a $\bQ$-basis of $R^\delta(\bH_k)_\bQ$ satisfying
  properties (P1) and (P2) above.
\end{enumerate}
\end{theorem}

The proof of (1) will occupy the rest of the section; we exhibit an
explicit basis $\{\pi_1,\dots,\pi_m\}$ for every pair $(\Psi,\delta).$
We remark that except for simply-laced systems and parameter $k\equiv
0$, or certain special values of $k$ in type $F_4$, the bases we give consist of
elements of $\Irr^\delta(\bH)_\el$ (for any parameter function). In
type $F_4$, for almost all values of the parameter $k$, the same is
true; however for a few special values of $k$ some of the $\pi_j$ may
become reducible.

\begin{proof}[Proof of (2)] 
Let $\overline R^\delta_\ell(\bH)$ be as in (\ref{e:graded-Groth}) and
$\overline R^\delta(\bH)=\oplus_\ell \overline R^\delta_\ell(\bH)$ be
the associated graded object. As $\bQ$-vector spaces
$R^\delta(\bH)_\bQ\cong \overline R^\delta(\bH)_\bQ,$ so it is
sufficient to construct a basis of $\overline R^\delta(\bH)_\bQ$ with
the desired properties.

Recall the set $\C I^\delta_0$ from
  section \ref{span-set setting}.  For every $J\in \C I^\delta_0$, the
  parabolic subalgebra $\bH_J$ decomposes as $\bH_J=\bH_J^\sem\otimes
  S(V^{W_J}).$ Then we have the following
  decomposition:
\begin{equation}
\begin{aligned}
\overline R^\delta(\bH)_\bQ&=\bigoplus_{J\in \C I^\delta_0}
i_J(\overline R^\delta_0(\bH_J^\sem)_\bQ\otimes
(S(V^{W_J\rtimes\<\delta\>})^{N_{W,\delta}(W_J)})^*),\\
\end{aligned}
\end{equation}
where the subgroup of unramified characters
$(S(V^{W_J\rtimes\<\delta\>})^{N_{W,\delta}(W_J)})^*$ can be
canonically identified with $N_{W,\delta}(W_J)\backslash
(V^\vee)^{W_J\rtimes\<\delta\>}.$ Let $\C B_{J,0}:=\C B( R^\delta_0(\bH_J^\sem))$
be the basis of $R^\delta_0(\bH_J^\sem)$ given by (1). Then the
desired basis for $\overline R^\delta(\bH)_\bQ$ (and so of
$R^\delta(\bH)_\bQ$) is 
\begin{equation}
\left\{ i_J(\pi\otimes \chi_\nu):\ J\in\C I^\delta_0,\ \pi\in \C B_{J,0},\
\nu\in N_{W,\delta}(W_J)\backslash
(V^\vee)^{W_J\rtimes\<\delta\>} \right\}.
\end{equation}

\end{proof}

\subsection{$A_{n-1}$}\label{sec:9.1} Let $\bH_{n,k}^A$ be the graded Hecke
algebra of $GL(n)$ with (constant) parameter $k$ (for simplicity of formulas we consider $GL(n)$
rather than $SL(n)$) generated by $w\in S_n$ and
$\{\ep_1,\dots,\ep_n\}$. Let $s_{i,j}$ denote the reflection in the
root $\ep_i-\ep_j.$ As it is well-known, there is a surjective algebra
morphism
\begin{equation}
\phi_k: \bH_{n,k}^A\to \bC[S_n],\quad \begin{cases}\phi_k(w)=w,\ w\in
  S_n;\\ \phi_k(\ep_i)=k(s_{i,i+1}+s_{i,i+2}+\dots+s_{i,n}),\ 1\le
  i<n;\\
\phi_k(\ep_n)=0.
\end{cases}
\end{equation}
Using the adjoint map $\phi_k^*: R(S_n)\to R(\bH_{n,k}^A)$, we can lift
every irreducible $S_n$-representation to an irreducible (in general
non-hermitian) $\bH_{n,k}^A$-module. If $\sigma$ is a partition of $n$
parameterizing an irreducible $S_n$-module, denote by
$\pi_{n,k}^A(\sigma)$ the resulting simple $\bH_{n,k}^A$-module.

There is a single elliptic conjugacy class in $S_n$, the class of
$n$-cycles. The space $\overline R_0(\bH_{n,k}^A)$ is one dimensional
spanned by the class of the trivial $\bH_{n,k}^A$-module
$\pi_{n,k}((n)).$ 

\subsection{${}^2A_{n-1}$}Let $\delta$ be the automorphism of order $2$ of the Dynkin diagram of
type $A_{n-1}.$ The elliptic $\delta$-twisted conjugacy classes in
$S_n$ are in one-to-one correspondence with partitions of $n$ where
every part is odd, see \cite[\S 7.14]{He07}. Every irreducible
$S_n$-representation is $\delta$-stable, i.e., $\Irr^\delta S_n=\Irr
S_n$. The representations $\pi_{n,k}(\sigma)$ constructed in section
\ref{sec:9.1} may seem therefore like good candidates for constructing
a basis in $\overline R_0^\delta(\bH^A_{n,k})$, but the problem is that
they are not typically $\delta$-invariant. This is because $\delta$
maps an irreducible $\bH$-module to its contragredient, and the
modules $\pi_{n,k}(\sigma)$ are not self-contragredient in general. 
A basis of $\overline R_0^\delta(\bH_{n,k}^A)$ is instead given by the
unitary induced modules 
\begin{equation}
\{\Ind_{\bH_{J(\sigma),k}^A}^{\bH_{n,k}^A}(\triv):\ \sigma \text{
  partition of $n$ into distinct parts}\},
\end{equation}
where $J(\sigma)$ is the subset of $I$ corresponding to the partition
$\sigma$; more precisely, if $\sigma=(n_1,n_2,\dots,n_\ell)$, then
$\bH_{J(\sigma),k}^A=\prod_{i=1}^\ell \bH_{n_i,k}^A.$

\subsection{$B_n/C_n$}\label{sec:9.2} The set of elliptic conjugacy
classes in $W(B_n)$ is in one-to-one correspondence with partitions of
$n$, \cite{Ca2}. For every partition $\lambda=(\lambda_1,\dots,\lambda_k)$ of $n$, let $w_\lambda$ be a
representative of the corresponding elliptic conjugacy class,
explicitly, $w_\lambda$ is a Coxeter element for
$W(B_{\lambda_1})\times\dots\times W(B_{\lambda_k}).$

From Definition \ref{d:graded}, one sees that there is an
isomorphism between the graded Hecke algebra of type $B_n$ with
parameters $k_1$ on the long roots and $k_2$ on the short roots, and
the graded Hecke algebra of type $C_n$ with parameters $k_1$ on the
short roots and $2k_2$ on the long roots. Because of this isomorphism,
we consider only the graded Hecke algebra of type $B_n$ (with
arbitrary parameters).

Let $\bH_{n,k_1,k_2}^B$ be the graded Hecke algebra of type $B_n$ with
parameter $k_1$ on the long roots and parameter $k_2$ on the short
roots. Let $W(B_n)$ be the Weyl group of type $B_n$ generated by
simple reflections $s_{i,i+1}$ for the roots $\ep_i-\ep_{i+1}$ and
$s_n$ for the root $\ep_n$.
In \cite{CK}, for every partition $\sigma$ of $n$, one
constructed an irreducible unitary $\bH_{n,k_1,k_2}^B$-module
$\pi^B_{n,k_1,k_2}(\sigma)$ via the assignment
\begin{equation}
\pi^B_{n,k_1,k_2}(\sigma):\begin{cases} s_{i,i+1}\mapsto
  \pi^A_{n,k_1}(\sigma)(s_{i,i+1}),\ 1\le i\le n-1;\\
s_n\mapsto \text{Id};\\
\ep_i\mapsto k_2\text{Id}+\pi^A_{n,k_1}(\sigma)(\ep_i),\ 1\le i\le n.
\end{cases}
\end{equation}
By definition $\pi^B_{n,k_1,k_2}(\sigma)|_{W(B_n)}=\sigma\times 0,$ in the
bipartition notation for irreducible $W(B_n)$-representations (\cite{Ca}).

An equivalent way to define $\pi^B_{n,k_1,k_2}(\sigma)$ is by the
assignment
\begin{equation}
\begin{aligned}
&w\mapsto (\sigma\times 0)(w), \quad w\in W(B_n);\\
&\wti\om\mapsto 0,\quad \om\in V=\bC\<\ep_1,\dots,\ep_n\>;
\end{aligned}
\end{equation}
here $\wti\om$ is defined as in (\ref{e:omtilde}) and it depends
linearly in $k_1,k_2$.

From the explicit description of representatives $w_\lambda$, it is
obvious that the matrix $(\Tr
(w_\lambda,\pi_{n,k_1,k_2}^B(\sigma)))_{\lambda,\sigma}$ equals the
character table of $S_n$,
$(\Tr((\lambda),\sigma))_{\lambda,\sigma}$, and thus it is invertible.

\subsection{$D_n$ and ${}^2D_n$}\label{sec:9.3}
Let $\bH_{n,k}^D$ be the graded Hecke algebra of type $D_n$ with constant
parameter $k$. Then $\bH_{n,k,0}^B\cong \bH_{n,k}^D\rtimes
\langle\delta\rangle,$ where $\delta$ is the automorphism of order $2$
of $D_n$ interchanging the simple roots $\ep_{n-1}-\ep_n$ and
$\ep_{n-1}+\ep_n.$

In particular, $\overline R_0(\bH_{n,k,0}^B)\cong \overline
R_0(\bH_{n,k}^D)\oplus\overline R_0^\delta(\bH_{n,k}^D),$ and
therefore this case follows from the general type $B_n$ case using
that the number of elliptic conjugacy classes in $W(D_n)$ plus the
number of $\delta$-twisted elliptic conjugacy classes in $W(D_n)$
equals the number of elliptic conjugacy classes in $W(B_n)$ (see
\cite[\S 7.20]{He07}).

Denote the restriction of  the module $\pi^B_{n,k,0}(\sigma)$ constructed before to
$\bH^D_{n,k}$ by $\pi^D_{n,k}(\sigma).$ Since $(\sigma\times 0)|_{W(D_n)}$ is irreducible, it follows that
$\pi^D_{n,k}(\sigma)$ is a unitary irreducible $\bH^D_{n,k}$-module.

\subsection{${}^3D_4$}\label{sec:D4Z3} Let $\delta$ be the
automorphism of order $3$ of the Dynkin diagram of type $D_4.$ From
\cite[Lemma 7.21]{He07}, there are
four conjugacy classes of $\delta$-twisted elliptic elements in
$W(D_4).$

A set of four linearly independent $\delta$-fixed $\bH(D_4)$-modules
in $\overline R^\delta_0(\bH)$ is formed of
\begin{equation}
\pi^D_{4,k}(\sigma),\text{ where $\sigma$ is any of the partitions
}(4),~(31),~(22), \quad \text{ and } \pi',
\end{equation}
where 
\begin{equation}
\Ind_{\bH(GL(3))}^{\bH(D_4)}\pi^A_{3,k}((3))=\pi'\oplus
\pi_{4,k}^D((22)),\text{ as }\bH(D_4)\text{-modules}.
\end{equation}
Notice that this induced module is zero in $\overline R_0(\bH(D_4))$,
but since $A_2\subset D_4$ is not fixed under $\delta$, it contributes
to $\overline R_0^\delta(\bH(D_4)).$

\subsection{$G_2$}\label{sec:9.4} For exceptional groups, we use the
labeling of irreducible $W$-representations from \cite{Ca}. 

There are
three elliptic conjugacy classes in $W(G_2).$ The irreducible
$W(G_2)$-representations $\phi_{1,0}$, $\phi_{1,6}$, $\phi'_{1,3}$,
$\phi''_{1,3}$ and $\phi_{2,2}$ can be extended to unitary
$\bH$-modules by letting $\wti\om$ act by $0$ for all $\om$, see
\cite{BM2}. In $\overline R_0(\bH)$, $\phi_{1,0}$ and $\phi_{1,6}$ are
in the same class, and so are $\phi'_{1,3}$ and $\phi''_{1,3}.$ The
set of $\bH$-modules supported on
\begin{equation}
\phi_{1,0},\ \phi'_{1,3},\ \phi_{2,2}
\end{equation}
is linearly independent in $\overline R_0(\bH).$ The table containing
the modules with their central characters is Table \ref{t:G2}. The
central characters are the residual central characters from
\cite{Op}. The $\om_i$ are the fundamental weights and the parameters
$k,k'$ are on long and short roots, respectively.

\begin{table}[h]
\caption{$G_2$\label{t:G2}}
\begin{tabular}{|c|c|}
\hline
central character &$W$-structure\\

\hline
$k\om_1+k'\om_2$ &$\phi_{1,0}$\\
\hline
$k\om_1+(-k+k')\om_2$ &$\phi'_{1,3}$\\
\hline
$k\om_1+\frac 12 (-k+k')$ &$\phi_{2,2}$\\
\hline
\end{tabular}
\end{table}

\subsection{$F_4$}\label{sec:9.5}
There are nine elliptic conjugacy classes in $W(F_4).$ There are nine
irreducible $W(F_4)$-representations that can be lifted to
$\bH$-modules by letting $\wti\om$ act by zero. They form only five
independent classes in $\overline R_0(\bH)$, for which a set of
representatives is
\begin{equation}
\phi_{1,0},\ \phi'_{1,12},\ \phi'_{2,4},\ \phi''_{2,4},\ \phi_{4,8}.
\end{equation}
There remain four modules to be specified. If an $\bH$-module is not
supported on a single $W$-type, the next simplest construction is when the
module is supported on a $W$-representation
$\sigma'+\sigma,$ where $\sigma$ and
$\sigma'$ are irreducible
$W(F_4)$-representations and $\Hom_{W}[\sigma',\sigma\otimes\refl]\neq
0.$ . We are able to construct the remaining four
modules in this way: $\phi_{4,1}+\phi_{1,0}$, $\phi'_{4,7}+\phi'_{1,12}$ and 
$\phi'_{8,3}+\phi''_{2,4}$ and $\phi_{6,6}''+\phi_{4,2}$. 
Notice that the first of them, supported on
$\phi_{4,1}+\phi_{1,0}$ (``the affine reflection representation'') also falls
into the framework of the $W$-graph
construction of \cite{L3} in the setting of affine Hecke
algebras. 
In the last module, we remark that $\phi_{6,6}''$ is the second exterior
power of the reflection representation $\phi_{4,2}$.

These
results are tabulated in Table \ref{t:F4}. The central characters are
the residual ones from \cite{Op}. The convention is that $\om_i$ are
the fundamental weights and $k$ is the parameter on the long roots,
while $k'$ is the parameter on the short roots. In each of the nine
families of representations constructed above, the actions of $\om_i$
(equivalently $\wti\om_i$) are linear in $k$ and $k'.$

\begin{table}[h]
\caption{$F_4$\label{t:F4}}
\begin{tabular}{|c|c|}
\hline
central character &$W$-structure\\

\hline
$k\om_1+k\om_2+k'\om_3+k'\om_4$ &$\phi_{1,0}$\\
\hline
$k\om_1+k\om_2+(-k+k')\om_3+k'\om_4$ &$\phi''_{2,4}$\\
\hline
$k\om_1+k\om_2+(-k+k')\om_3+k\om_4$ &$\phi_{4,1}+\phi_{1,0}$\\
\hline
$k\om_1+k\om_2+(-2k+k')\om_3+k'\om_4$ &$\phi'_{1,12}$\\
\hline
$k\om_1+k\om_2+(-2k+k')\om_3+2k\om_4$ &$\phi'_{8,3}+\phi''_{2,4}$\\
\hline
$k\om_1+k\om_2+(-2k+k')\om_3+k\om_4$ &$\phi'_{4,7}+\phi'_{1,12}$\\
\hline
$k\om_1+k\om_2+(-2k+k')\om_3+(3k-k')\om_4$ &$\phi'_{2,4}$\\
\hline
$k\om_2+(-k+k')\om_4$ &$\phi_{4,8}$\\
&$\phi_{6,6}''+\phi_{4,1}$\\
\hline
\end{tabular}
\end{table}

\subsection{$E$}\label{sec:9.6} This case (as well as every equal
parameter untwisted $\bH$) is covered by the results of \cite{Re} in
terms of the geometric classification of $\bH$-modules \cite{L2}. More
precisely, a set of representatives for a basis of $\overline R_0(\bH)$ is given
by (tempered) modules $\pi(e,\phi)$ parameterized geometrically by pairs
\begin{equation}
(e,\phi),
\end{equation}
where
\begin{enumerate}
\item[(i)] $e$ ranges over a set of representatives for quasi-distinguished (\cite[(3.2.2)]{Re})
nilpotent orbits in the Lie algebra corresponding to the root system;
\item[(ii)] $\phi$ is an irreducible representation of 
$A(e)$, the group of components of the centralizer of $e$ in the
associated adjoint Lie group, such that
\item[(iii)] $\phi$ is a representative for its class in $\overline R_0(A(e))$, the
elliptic representation space for $A(e)$ defined in \cite[\S 3.2]{Re}.
\end{enumerate}
Concretely, for type $E$, for every distinguished $e$ (in the sense of
Bala-Carter \cite{Ca}), one allows every $\phi$ of ``Springer type'',
while for $e$ quasi-distinguished, but not distinguished, one allows
only $\phi=1.$ For example, for $E_6$, a basis is labeled by the five pairs
\begin{equation}
(E_6,1),\ (E_6(a_1),1),\ (E_6(a_3),1),\ (E_6(a_3),\sgn),\ (D_4(a_1),1).
\end{equation}

\subsection{${}^2E_6$} Let $\delta$ be the automorphism of order $2$
of the Dynkin diagram of type $E_6$.  According to \cite[Table III]{GKP}, there are
nine $\delta$-twisted elliptic conjugacy classes in $W(E_6).$ A set of
nine linear independent (in $\overline R_0^\delta(\bH(E_6))$)
$\delta$-fixed irreducible modules is formed of the tempered modules
$\pi(e,\phi)$ (see section \ref{sec:9.6}), for the pairs
\begin{equation}
\begin{aligned}
&(E_6,1),\ (E_6(a_1),1),\ (E_6(a_3),1),\ (E_6(a_3),\sgn),\ (D_5,1),\\
&(D_5(a_1),1),\ (A_4+A_1,1),\ (D_4(a_1),1),\ (D_4(a_1),\refl).
\end{aligned}
\end{equation}
Notice that the subsets of type $D_5$ and $A_4+A_1$ are not
$\delta$-stable. The component group for $D_4(a_1)$ is $A(e)=S_3.$

\subsection*{Acknowledgements} After a first version of this paper was
posted at \texttt{arxiv.org}, M. Solleveld pointed out to us that he proved similar results for real parameter in \cite[Proposition 2.5 \& Theorem 3.1]{So2}.

\smallskip

The first named author thanks the Department of Mathematics
at Hong Kong University of Science and Technology for its
hospitality and support during his visit, when part of the research for this paper
was completed. D. C. was supported in part by NSF-DMS 0968065 and X. H. was supported in part by HKRGC grant 602011. 

%%%%%%%%%%%%%%%%%%%%%%%%%%%%%%%%%%%%%%%%%%%%%%%%%%%%%%%%%%

\ifx\undefined\bysame
\newcommand{\bysame}{\leavevmode\hbox to3em{\hrulefill}\,}
\fi

\end{document}